\documentclass[11pt,twoside]{amsart}

\usepackage{color} 
\usepackage{amssymb,amscd, amsthm}
\usepackage{setspace}
\usepackage{comment}
\usepackage{hyperref, cleveref}
\usepackage{url} 
\def\opn#1#2{\def#1{\operatorname{#2}}} 
\opn\Im{Im}
\opn\rank{rank}
\opn\ini{in}
\opn\PF{PF}
\opn\RF{RF}
\opn\sgn{sgn}
\opn\Fr{Fr}
\opn\Ap{Ap}
\opn\supp{supp}
\opn\mult{mult}
\opn\ord{ord}
\opn\alt{alt}
\opn\Alt{Alt}
\opn\Sym{Sym}
\opn\relint{relint}
\opn\Diag{Diag}
\opn\det{det}
\opn\diag{diag}
\opn\embdim{emb\,dim}
\newcommand{\Z}{\mathbb Z}
\newcommand{\N}{\mathbb N}
\newcommand{\MM}{\mathbb M}
\newcommand{\YY}{\mathbb Y}

\newcommand{\cS}{\mathcal{S}}
\newcommand{\cC}{\mathcal{C}}

\def\mm{\mathfrak m}

\def\bv{{\mathbf v}}
\def\bw{{\mathbf w}}
\def\by{{\mathbf y}}
\def\ba{{\mathbf a}}
\def\bx{{\mathbf x}}
\def\be{{\mathbf e}}
\def\bm{{\mathbf m}}

\def\bbM{{\mathbb M}}
\def\al{\alpha}

\newtheorem{Theorem}{Theorem}[section]
\newtheorem{thm}[Theorem]{Theorem}
\newtheorem{mainthm}{Theorem}
\newtheorem{prop}[Theorem]{Proposition}
\newtheorem{lem}[Theorem]{Lemma}
\newtheorem{cor}[Theorem]{Corollary}

\theoremstyle{definition}
\newtheorem{rem}[Theorem]{Remark}
\newtheorem{ex}[Theorem]{Example}
\newtheorem{defn}[Theorem]{Definition}
\newtheorem{note}[Theorem]{Notation}
 
\numberwithin{equation}{section}

\begin{document}
\title[
Almost symmetric almost complete intersection numerical semigroups
]{
The structure of almost symmetric almost complete intersection numerical semigroups} 

\author[K. Eto, N. Matsuoka, A. Moscariello, T. Numata, A. Sammartano, K. Watanabe]{Kazufumi Eto,
Naoyuki Matsuoka,
Alessio Moscariello,
Takahiro Numata,
Alessio Sammartano,
Kei-ichi Watanabe
}

\address[Kazufumi Eto]{Department of Mathematics, Nippon Institute of Technology, 
Miyashiro, Saitama, Japan 345-8501}
\email{etou@nit.ac.jp}
\address[Alessio Moscariello]{
Department of Mathematics and Scientific Computing, University of Graz, NAWI Graz,
Heinrichstra{\ss}e 36,
8010 Graz}
\email{alessio.moscariello@uni-graz.at}

\address[Naoyuki Matsuoka]{Department of Mathematics, School of Science and Technology, Meiji University,
1-1-1 Higashi-mita, Tama-ku, Kawasaki, Japan 214-8571
}
\email{naomatsu@meiji.ac.jp}

\address[Takahiro Numata]{School of Science and Technology, Meiji University,
1-1-1 Higashi-mita, Tama-ku, Kawasaki, Japan 214-8571
}
\email{takahiro.n0913@gmail.com}

\address[Alessio Sammartano]{Dipartimento di Matematica,
Politecnico di Milano,
Milano, 20133, Italy}
\email{alessio.sammartano@polimi.it}

\address[Kei-ichi Watanabe]{Department of Mathematics, College of Humanities and Sciences, 
Nihon University, Setagaya-ku, Tokyo, 156-8550, Japan and 
Organization for the Strategic Coordination of Research and Intellectual Properties, Meiji University
}
\email{watnbkei@gmail.com}


\begin{abstract} 
We prove a structure theorem for numerical semigroups $H$ that are almost symmetric and almost complete intersections. 
Specifically, we show that a row-factorization (RF) matrix of $H$ must possess a highly regular structure, 
which we call a \emph{cascade matrix}.
Consequently,  the defining ideal $I_H$ 
 of the associated semigroup ring $\Bbbk[H]$ also exhibits 
a highly regular structure, derived from this cascade matrix.
Moreover, both the RF-matrix and the binomial minimal generating set of $I_H$ are unique, two properties that are remarkably rare for general numerical semigroups.
Conversely, 
we show that this structure completely characterizes almost symmetric almost complete intersection numerical semigroups: 
to every cascade matrix $M$ we associate a submonoid $H \subseteq \N$ and, whenever
this $H$ is a numerical semigroup, we prove that it is  pseudo-symmetric, almost complete intersection, and has $M$ as RF-matrix.
In this way, for each odd integer $e$, 
we establish a new way of constructing a numerical semigroup $H$ from an  $e$-tuple of positive integers.

As a consequence of our study, we obtain several key  results.
\begin{enumerate}
    \item Rigidity: 
    if an almost complete intersection  semigroup is almost symmetric, then  it is forced to have odd embedding dimension and to be pseudo-symmetric.
This result can be regarded
as the ``next step'' after Kunz's theorem, which states that  an almost complete intersection semigroup is never symmetric.
\item Huneke--Ulrich structure: the toric ideal $I_H$  can be realized as the ideal of entries of a skew-symmetric matrix multiplied by a vector of variables. 
This implies that $I_H$ is in the linkage class of complete intersections.
 \item Cascade polynomials: for each odd positive integer $e$, we construct a multivariate squarefree polynomial $P_e$ with integer coefficients, arising from a cascade matrix of variables. We provide an enumerative interpretation of its coefficients, thereby proving their non-negativity.
 
 \item Herzog--Watanabe question:
 we prove that
 every  minimal relation of an arbitrary numerical semigroup $H$
 can be obtained by subtracting two rows in some RF-matrix of $H$,
 affirmatively answering a 2019 question by Herzog and Watanabe.
\end{enumerate}

\end{abstract}

\maketitle

\section{Introduction}

A classical theme in commutative algebra is understanding how  the structure of an ideal 
is dictated by algebraic or homological conditions such as the Cohen--Macaulay, Gorenstein, or complete intersection properties, and their generalizations.
Notable results in this direction include the Hilbert--Burch \cite{Burch}
and Buchsbaum--Eisenbud \cite{BE} theorems, characterizing the structure of Cohen–Macaulay ideals of codimension 2 and Gorenstein ideals of codimension 3, respectively; subsequent work has yielded structure theorems for several classes of ideals of low codimension.
For an overview of the theory and some very recent developments, 
see \cite{CGNW,GCNW} and the references therein.

It is  natural to investigate this general problem for classes of ideals satisfying additional geometric or combinatorial properties,
such as   toric ideals.
In this paper, we investigate the structure of toric ideals of numerical semigroups $H \subseteq \mathbb{N}$.
In this setting, the  arithmetic framework of the semigroup often leads to sharper structure theorems, and sometimes to entirely new phenomena.
For example, 
Herzog's result for 3-generated  numerical semigroups  \cite{H} is a refined 
version of the Hilbert--Burch theorem, 
and Bresinsky's theorem for 4-generated symmetric numerical semigroups \cite{B75}
refines the Buchsbaum--Eisenbud theorem.
For an account of the current state of the art and key open problems regarding toric ideals of numerical semigroups, see \cite{MS}.

For a numerical semigroup $H $ minimally generated by $e$ integers, 
its toric ideal $I_H$ has codimension $e-1$, and hence requires at least $e-1$ minimal generators.
The ideal $I_H$ (or the semigroup $H$) is called a \emph{complete intersection} if $I_H$ is generated by $e-1$ elements, and  an \emph{almost complete intersection} if it is generated by  $e$ elements.
A numerical semigroup is \emph{almost symmetric} if its semigroup ring is almost Gorenstein, a generalization of the Gorenstein property to arbitrary Cohen--Macaulay type introduced in \cite{BF} that has received considerable attention,
see e.g. \cite{E,GTT,HW,N};
the reader can find the  definition in Section~\ref{SubsectionNumericalSemigroups}.

The main result of this work is a structure theorem for  toric ideals of   numerical semigroups that are 
 almost complete intersections and almost symmetric, see Theorem \ref{mainTh} below.
We obtain  a complete description of these toric ideals,  
showing that they possess a highly structured and unique minimal set of binomial generators.
In turn, 
these binomials arise from the data of a \emph{row-factorization (RF) matrix}, 
a gadget introduced in \cite{Mo} that  encodes factorizations of pseudo-Frobenius numbers of $H$;
this matrix too is highly structured and unique for the semigroups considered here.

The existence of such a description in unrestricted codimension is somewhat surprising, since
 most of the known 
 structure theorems for almost complete intersections rely heavily on low  codimension assumptions.
Specifically, the  case of codimension three is  completely solved in the classical paper \cite{BE}, and subsequent works have tackled the case of codimension 4 \cite{Kustin, Kustin93,Palmer}.

We now give the exact statement of our main result;
see Section~\ref{SectionBasicConcepts} for precise definitions.

\begin{mainthm}\label{mainTh} Let $H=\langle n_1, n_2, \ldots, n_e\rangle$ be  an almost symmetric almost complete intersection numerical semigroup of embedding dimension  $e\ge 3$.
Then 
\begin{enumerate}
\item $H$ is  pseudo-symmetric;
\item $e=2e'+1$ is odd;
\item up to changing the order of $ n_1, n_2, \ldots, n_e$,
the toric ideal 
$I_H$ is generated by the binomials
\[
x_i^{\al_i} - x_{i+1}{x_{i+e'+1}^{\al_{i+e'+1}-1}} \qquad  i = 1,2 ,\ldots, e \]
where
$\al_i$ is the least integer such that $\alpha_i n_i$ admits another factorization in $H$,
and where we identify $x_j = x_{j-e}$ and $\al_j =\al_{j-e}$ if $j> e$.
Moreover, this is the unique minimal set of binomial generators of $I_H$;

\item the pseudo-Frobenius number $f=\frac{\Fr(H)}{2}$ has an RF-matrix
$ (m_{ij})_{1\le i, j \le e}$ 
\[ m_{ij} = \left\{
\begin{array}{ll}
-1 & \text{ if  } i=j, \\
0  &  \text{ if }  0 < j-i \le e', \\
\al_j-1 & \text{ otherwise}. 
\end{array}
\right. 
\]
where  $j-i$ is considered modulo $e$.  
Moreover, this is the unique RF-matrix of $f$.
\end{enumerate}
Conversely, let $\ba = (a_1, \ldots , a_e)$ be positive integers, and consider the matrix  $\bbM = (m_{ij})$ defined as
\[ m_{ij} = \left\{
\begin{array}{ll}
-1 & \text{ if  } i=j, \\
0  &  \text{ if }  0 < j-i \le e', \\
a_j & \text{ otherwise}, 
\end{array}
\right.  \]
where the difference $j-i$ is  considered modulo $e$.
Assume $\ba \ne(1,1,1)$.
Let  
\[(n_1, \ldots , n_e)^\intercal = \det(\bbM) \cdot \bbM^{-1}{\bf 1}_e,\]
 where we denote by {\bf 1} the column vector of length $e$, 
whose every entry is $1$.
Then
\begin{enumerate}
    \item $\det(\bbM)>0$,
\item $n_i > 0$ for every $i=1,\ldots,e$,
\end{enumerate}
and,  if in addition we have $\gcd(n_1, n_2, \ldots, n_e)=1$, then
\begin{enumerate}
\setcounter{enumi}{2}
\item $H = \langle n_1, n_2, \ldots, n_e \rangle$
is a  numerical semigroup of embedding dimension $e$,
\item $H$ is pseudo-symmetric, in particular, it is almost symmetric,
\item $H$ is an almost complete intersection,
\item $\bbM$ is an RF-matrix for $H$.
\end{enumerate}
\end{mainthm}

For example, the semigroup $H = \langle 64,75,91,95,97,99,107\rangle$ satisfies the hypotheses of Theorem \ref{mainTh}, and  
is  pseudo-symmetric with pseudo-Frobenius numbers $\PF(H)=\{239,478\}$. 
Reordering the generators as  $\{91,75,97,99,107,64,95\}$, the pseudo-Frobenius number $f=239$ has a unique row-factorization matrix
$$
\RF(239)=\begin{pmatrix}
-1 & 0 & 0 & 0 & 1 & 2 & 1 \\
1 & -1 & 0 & 0 & 0 & 2 & 1\\
1 & 2 & -1 & 0 & 0 & 0 & 1\\
1 & 2 & 1 & -1 & 0 & 0 & 0\\
0 & 2 & 1 & 1 & -1 & 0 & 0\\
0 & 0 & 1 & 1 & 1 & -1 & 0 \\
0 & 0 & 0 & 1 & 1 & 2 & -1
\end{pmatrix}.
$$
The rows  are obtained from factorizations of  $239+91, \ldots, 239+95$ in $H$.
From this matrix we read the generators of the toric ideal
$$
 x_1^2 - x_2x_5,\, \,
 x_2^3-x_3x_6^2,\,\,
 x_3^2-x_4x_7,\,\,
 x_4^2-x_5x_1,\,\,
 x_5^2-x_6x_2^2,\,\,
 x_6^3-x_7x_3,\,\,
 x_7^2-x_1x_4,
$$
with  each binomial corresponding to the difference of two consecutive rows of $\RF(239)$.

The case $\ba = (1,1,1) $ is a genuine exception to this construction, see Example \ref{ExException111}.

Let us now comment on the content of Theorem \ref{mainTh}, emphasizing several unexpected features.

A consequence of  Theorem \ref{mainTh} is that all the relations among $n_1, n_2, \ldots, n_e$ are generated by the simplest kind of relations, namely, those expressing a multiple of each $n_i$ as a combination of the remaining generators. 
A toric ideal with this property is called \emph{critical}. 
We investigate this concept in detail in Section \ref{SectionCriticalIdeals}.
To emphasize the exceptionality and  sharpness of  Theorem \ref{mainTh}, we point out that both the critical condition and the uniqueness of minimal binomial generators may fail for complete intersection numerical semigroups (which are automatically symmetric and thus almost symmetric),
see Examples \ref{symnotcrit} and \ref{CInotunique}.

Other unexpected consequences are the following rigidity phenomena.
First, there exist no almost symmetric almost complete intersection semigroups of \emph{even} embedding dimension.
Second, when they exist, they are forced to be pseudo-symmetric, in particular, to have type equal to two.
The latter result is unexpected, since, individually, neither
the almost symmetric nor the almost complete intersection condition imposes any obvious restriction on the type\footnote{However, see \cite[Question 13]{Mo} and \cite[Section 5]{MS} for subtle implications of the almost symmetric condition on the type.}. 
In fact, 
this situation resembles a classical theorem of Kunz \cite{Ku74}, 
which, when specialized to numerical semigroups, 
states that  an almost complete intersection semigroup is never symmetric;
our result can then be interpreted as a next step after Kunz's theorem.

From the shape of the binomials discovered in Theorem \ref{mainTh}, 
we  deduce in Corollary \ref{CorollaryHU} that $I_H$ is a \emph{Huneke-Ulrich almost complete intersection}, a  class of ideals investigated first in \cite{HU}, see also \cite{Kustin95}.
This bears an important consequence in terms of linkage: the ideal $I_H$ is licci, that is,
it belongs to the linkage class of a complete intersection.

The converse direction of 
Theorem \ref{mainTh} uncovers  a new way to construct  semigroups from a tuple of integers $(a_1, \ldots, a_e) $, which is interesting in its own right.
For example, 
we show that the Frobenius number of a semigroup constructed in this way is twice the determinant of the matrix $\bbM$, see Proposition \ref{PropPFfromM}.
A natural question arises:
what are the tuples of integers that actually correspond to numerical semigroups? 
In other words, when are the entries of the vector $\det(\bbM) \cdot \bbM^{-1}{\bf 1}_e$ coprime?
This condition is rather  subtle: for instance,  Example \ref{gcd hyp} shows that the answer is  sensitive to the order of $a_1, \ldots, a_e$.

In Section~\ref{SectionCascade} we study matrices of variables with the same shape as the  $\bbM$ of Theorem \ref{mainTh}, which we call \emph{cascade matrices}.
We focus in particular on their determinant, which we call the \emph{cascade polynomial}, 
and show that it is a squarefree polynomial all of 
whose non-constant coefficients are non-negative integers.
Because non-negativity of the coefficients does not follow immediately from the definition, 
a trait shared by many notable polynomials in algebraic combinatorics, 
the cascade polynomial may be an object of independent interest.

Finally, en route to proving Theorem \ref{mainTh}, we affirmatively settle a question of Herzog and Watanabe for arbitrary numerical semigroups \cite[Question 3.6]{HW}.
They asked  whether the toric ideal $I_H$ is always generated by binomials that can be obtained as differences of two rows of some RF-matrix of $H$.
Specifically,
one of the key steps required in the proof of Theorem \ref{mainTh} is to show that all minimal binomial generators of $I_H$ arise from a \emph{single} RF-matrix of $H$,
which we establish in Proposition \ref{RF(f)}.
This phenomenon is tied  to the nature of the special semigroups we consider, and  cannot hold in greater generality, since the number of generators of $I_H$ can be arbitrarily large compared to  the size of an RF-matrix.
However, 
by adapting similar ideas as those in the proof of Proposition \ref{RF(f)}, we obtain a short proof  of the Herzog--Watanabe question in general, which we present separately in Section~\ref{SectionHW}.

\subsection*{Acknowledgements}
The authors would like to thank Alessio D'Alì, Pedro A. García Sánchez, and Andreas Reinhart for helpful conversations.
A. M., A. S., and K. W. are grateful to the organizers of the conference \emph{Modules and Rings: Recent Developments in Commutative Algebra}, held in Genoa in June 2025, 
where initial ideas coalesced to form the basis of this paper.
Experiments with Macaulay2 \cite{M2} and GAP \cite{GAP}, using the NumericalSgps package \cite{DGM}, have provided valuable information.

\subsection*{Funding}
N. M. was supported by JSPS Grant-in-Aid 
for Scientific Research (C) Grant Number JP25K06941.
A. M. was supported by the Austrian Science Fund FWF, Project PAT 9756623.
A. S.  was  supported by 
PRIN 2022K48YYP \emph{Unirationality, Hilbert schemes, and singularities},
and is a member of  INdAM – GNSAGA.
K. W.  was partially supported by JSPS Grant-in-Aid 
for Scientific Research (C) Grant Number 23K03040.

\subsection*{AI disclosure}
LLMs were used solely for language editing and proofreading the final manuscript.

\section{Basic concepts}\label{SectionBasicConcepts}

In this section, we fix the notation and recall the 
 definitions  which will be used in this paper.

\subsection{Numerical semigroups}\label{SubsectionNumericalSemigroups}
Let $\N$ denote the set of non-negative integers.
A numerical semigroup is an additive submonoid of $\N$ with finite complement.
Throughout this paper, let $H$ denote a numerical semigroup.
There exist elements  $n_1,\ldots,n_e\in H$ such that every $h\in H$ can be written as $h=\sum_{i=1}^e a_in_i$ with  $a_i \in \N$. 
Such set   $\{n_1,\ldots,n_e\}\subset H$ is called a {\em set of generators} of $H$, and we  write $H=\langle n_1,\ldots ,n_e\rangle$,
and an expression $h=\sum_{i=1}^e a_in_i$ as above  is called a {\em factorization} of $h$.
More generally, for $n_1,\ldots,n_e\in \N$ we denote by $\langle n_1,\ldots ,n_e\rangle$ the submonoid of $\N$ generated by $n_1,\ldots,n_e$;
it is a numerical semigroup if and only if $\gcd(n_1,\ldots,n_e)=1$.
The set 
$\{n_1,\ldots,n_e\}$ is called a {\em minimal set of generators} of $H$ if no $n_i$ can be omitted to generate $H$,
in this case $e$ is called the {\em embedding dimension} of $H$.
Every submonoid of $\N$ has a unique
 minimal set of generators.
We  do not assume $n_1, \ldots, n_e $ to be in increasing order.

The {\em Frobenius number} of $H$ is $\Fr(H) = \max ( \Z \setminus H)$.
An element $f\in \Z\setminus H$ is called a  {\em pseudo-Frobenius} number
if $f+n_i\in H$ for all $i$.
The set of  pseudo-Frobenius numbers is  denoted by $\PF(H)$,
and its cardinality is called the {\em type} of $H$, denoted $t(H)$. 
Clearly, we have $\Fr(H) \in \PF(H)$.

The {\em Apery set} of $H$ with respect to an element $a\in H$ is
$\Ap(H,a) = \{ h\in H\;|\; h-a \not\in H\}.$
We have
$|\Ap(H,a)|=a$,  $0 \in \Ap(H,a)$, $n_i\in\Ap(H,a)$ whenever $n_i \ne a$,
and $\max\Ap(H,a) = a + \Fr(H)$.

We say that $H$ is {\em symmetric} if $\PF(H)=\{\Fr(H)\},$
equivalently, 
if
$ \Z = H \cup  \{  \Fr(H) -h \,|\,  h\in H\}$. 
If $H$ is symmetric, then $\Fr(H)$ is odd.
We say that  $H$ is {\em pseudo-symmetric} if $\Fr(H)$ is even and 
$\PF(H)=\{\Fr(H),  \Fr(H)/2 \}.$
We say that  $H$ is {\em almost symmetric} if for every $f \in \PF(H), f \ne \Fr(H),$ we have $\Fr(H)-f \in \PF(H)$;
this class includes both symmetric and pseudo-symmetric semigroups.

Given a numerical semigroup $H$, we define a partial order on $\Z$ by setting 
$b \leq_H c$ if  $c-b \in H$.
The maximal elements of $\Ap( H,a)$ with respect to this partial order are $f+a$ where $f \in\PF(H)$.

\subsection{Semigroup rings}
Fix a field $\Bbbk$,
and let $H$ 
be a numerical semigroup minimally generated by $\{n_1, \ldots, n_e\}$.
The semigroup ring of $H$ is $\Bbbk[H] = \Bbbk[t^{n_1},\ldots,t^{n_e}]\subseteq \Bbbk[t]$,
where $t$ is a variable.
It is a positively graded  $1$-dimensional Cohen--Macaulay domain,
with grading $\deg(t^h) = h$.

Let $S = \Bbbk[x_1, \ldots, x_e]$ be a polynomial ring, 
 graded by $\deg(x_i) = n_i$, and let $\mm = (x_1, \ldots, x_e)$. 
 We denote monomials of $S$ by 
$\bx^\ba = x_1^{a_1}\cdots x_e^{a_e}$, where $\ba = (a_1, \ldots, a_e) \in \N^e.
$

Consider the surjective graded ring homomorphism
$$
\Psi_H : S \to \Bbbk[H], \quad \Psi_H(x_i) = t^{n_i}.
$$
It induces an isomorphism $\Bbbk[H] \cong S/I_H$,
where $I_H= \ker(\Psi_H)$ is called the {\em toric ideal} of $H$ or the  {\em defining ideal} of $\Bbbk[H]$.
This ideal has codimension $e-1$, thus, its minimal number of generators $\mu(I_H)$
satisfies $\mu(I_H) \geq e-1$.
The ideal $I_H$, or the semigroup $H$,
is called a \emph{complete intersection} if $\mu(I_H)=e-1$,
and it is called an \emph{almost complete intersection} if $\mu(I_H)= e$.

We now describe this ideal explicitly in terms of the semigroup $H$.
Let ${\bf v}= (v_1,\ldots , v_e)\in \Z^e$.
Its {\em support} is
$\supp ({\bf v}) = \{ i \;|\; v_i\ne 0\}.$
Let $\bv^+, \bv^- \in \N^e$ be the unique vectors such that $\bv = \bv^+-\bv^-$ and $\supp(\bv^+)\cap \supp(\bv^-)=\emptyset$.
For such $\bv$,  define the binomial
\[
F(\bv) = \bx^{\bv^+} - \bx^{\bv^-}\in S.
\]
We further denote by $\be_1, \ldots, \be_e$ the standard basis vectors of the free abelian group $\Z^e$, and by  $\sigma_i(\bv)$ the $i$-th coordinate of a vector $\bv \in \Z^e$.

Consider the surjective group homomorphism
\[
\Phi_H : \Z^e \to \Z, \quad \Phi_H (a_1, \ldots , a_e) = \sum_{i=1}^e a_in_i.
\]  
The lattice $\ker(\Phi_H) \subseteq \Z^e$ has rank $e-1$.
The ideal $I_H$ is generated by the set $\{F(\bv) \,:\, \bv \in \ker(\Phi_H)\}$, and it is spanned as a $\Bbbk$-vector space by all binomials $x^\bv-x^\bw$ such that $\Phi_H(\bv)=\Phi_H(\bw)$.

For every $\bv \in \ker(\Phi_H)$,
the binomial $F(\bv)$ is homogeneous of degree 
$$
\deg(F(\bv)) = \sum_{\sigma_k(\bv)> 0} v_kn_k  =\sum_{\sigma_k(\bv)< 0} (-v_k)n_k\in H.
$$
We say that a binomial $F(\bv) \in I_H$ is {\em minimal} if it is part of some minimal set of generators of $I_H$, 
that is,
if its image in $I_H/\mm I_H$ is non-zero.

The following distinguished class of binomials in $I_H$,
encoding the simplest kind of relation among the generators of $H$,
plays a central role in this work.

\begin{defn}\label{DefCriticalBinomial}
For each $i =1, \ldots, e$   define 
\begin{equation*}
 \al_i = \min\big\{ a\in \Z_{>0}\; |\;  a n_i \in \langle n_1,\ldots ,n_{i-1},n_{i+1}, \ldots n_e\rangle\big\}.
\end{equation*}
A binomial in $I_H$ of the form 
\begin{equation}\label{EqCriticalXi}
x_i^{\al_i} - \prod_{j=1, j\ne i}^e x_j^{b_j}     
\end{equation}
is called a {\em critical binomial} with respect to $x_i$.
\end{defn}

\begin{rem}\label{RemBasicObservationsCritical}
We make some basic observations.
A critical binomial is necessarily minimal,
since the monomial  $x_i^{\alpha_i}$ cannot be in the support of any binomial in $\mm I_H$, by minimality of $\alpha_i$.
Conversely, any set of binomial generators of $I_H$ must contain critical binomials with respect to each $x_1, \ldots, x_e$,
since the binomial \eqref{EqCriticalXi} cannot be generated by  binomials not containing the monomial $x_i^{\alpha_i}$  in their supports.
A binomial may be critical with respect to two variables $x_i \ne x_j$, this happens precisely when $\alpha_i n_i = \alpha_j n_j$.
Finally, a critical binomial with respect to $x_i$ may not be unique, as 
it also depends on a choice of factorization $\sum_{j=1, j\ne i}^e b_jn_j$. 
\end{rem}

\subsection{Convention of the index set}\label{SubsectionCircle}
Throughout this paper,  the set $\{1, \ldots, e\}$ of the first $e$ positive integers often appears as index set for  variables or for a tuple of integers.
We always regard this set  as arranged in a circle.
This leads to the following slight abuse of notation: if we add or subtract indices and the result is outside this range, then we mean the unique integer in $\{1, \ldots, e\}$ congruent to the result modulo $e$.

\section{The defining ideal is generated by RF-relations}\label{SectionHW}

In this section, we introduce a special class of binomials in the defining ideal $I_H$ of a numerical semigroup $H$, called \emph{$\RF$-relations}, and we show that they always generate the ideal $I_H$.
This settles a question of Herzog--Watanabe.

Let $H$ 
be a numerical semigroup minimally generated by $\{n_1, \ldots, n_e\}$.
\begin{defn}
Let $f \in \PF(H)$ be a pseudo-Frobenius number.
We define  an $\RF$-matrix $\RF(f)$ to be a
 matrix $(m_{ij})_{1\leq i,j\leq e}$ such that 
 $m_{ii} = -1$ for every $1\le i\le e$, $m_{ij} \ge 0$ if $i \ne j$ and 
$\sum_{j=1}^e m_{ij}n_j = f$ for all $i$.
\end{defn}

The data in the $i$-th row of $\RF(f)$ is equivalent to the data of a factorization of $f+n_i \in H$.
We point out that 
an $\RF$-matrix $\RF(f)$ is not uniquely determined by $f$, as it also depends on choices of factorizations of $f+n_i \in H$.
In what follows, when $\RF(f)$ appears, we  assume that such factorizations have been chosen.

\begin{defn}\label{RF-rel} 
Let $f\in \PF(H)$ and $\RF(f)=(m_{h\ell})$ be an $\RF$-matrix.
For each $i\ne j$, let $\bm_i = \sum_{\ell=1}^e m_{i\ell}\be_\ell $ and
$\bm_j = \sum_{\ell=1}^e m_{j\ell}\be_\ell $ 
denote the $i$-th and $j$-th rows of $\RF(f)$, and let $\bv = \bm_i - \bm_j \in \Z^e$.
We call the binomial
$$
F_{ij}(\RF(f)) = F(\bv) = \prod_{m_{i\ell}-m_{j\ell}>0} x_\ell^{m_{i\ell}-m_{j\ell}} 
- \prod_{m_{i\ell}-m_{j\ell}<0} x_\ell^{m_{j\ell}-m_{i\ell}}
$$
an {\em $\RF$-relation} of $H$ associated to $f$.
\end{defn}
Observe that  indeed $F_{ij}(\RF(f))\in I_H$, since $\sum_{\ell =1}^e m_{i\ell}n_\ell = \sum_{\ell=1}^e m_{j\ell} n_\ell =f$.

\begin{ex}

Let $H=\langle5,6,7,8,9\rangle$.
Then $\PF(H)=\{1, 2,3,4\}$. 
The following are $\RF$-matrices associated to  $3 \in \PF(H)$
\[
\begin{pmatrix}
-1 & 0 & 0 & 1 & 0 \\
0 & -1 & 0 & 0 & 1 \\
2 & 0 & -1 & 0 & 0 \\
1 & 1 & 0 & - 1 & 0 \\
0 & 2 & 0 & 0 & -1\end{pmatrix}
\quad\text{and}\quad
\begin{pmatrix}
-1 & 0 & 0 & 1 & 0 \\
0 & -1 & 0 & 0 & 1 \\
2 & 0 & -1 & 0 & 0 \\
1 & 1 & 0 & - 1 & 0 \\
1 & 0 & 1 & 0 & -1
\end{pmatrix}.
\]
The bottom row in the two matrices arises from a different factorization $2\cdot 6 = 5+7$ of $3+9 \in H$.

Choosing the first matrix as $\RF(3)$, we have  $\bm_1=(-1,0,0,1,0)$ and $\bm_2=(0,-1,0,0,1)$, so $\bm_1-\bm_2=(-1,1,0,1,-1)$ yields the $\RF$-relation $F_{1,2}(\RF(3))=x_2x_4-x_1x_5$. 
\end{ex}

We now prove the main result of this section.

\begin{thm}\label{ThmRFRelations}
Let $H$ be a numerical semigroup minimally generated by $n_1, \ldots,n_e$.
Let $F(\bv) $ be a minimal binomial of $I_H$,
with $\bv= (v_1, \ldots, v_e)$,
and let $h = \deg(F(\bv))\in H$.
If $i,j$ are  indices such that $\sigma_i(\bv) > 0$ and $\sigma_j(\bv) < 0$, then $h-n_i-n_j \not \in H$. 
Moreover, if $f \in \PF(H)$ is such that $h-n_i-n_j \le_H f$, then $F(\bv)$ is an $\RF$-relation of $H$ associated to $f$.
\end{thm}

\begin{proof} 
Assume by contradiction that $h - n_i - n_j \in H$,
then there exists $\bw=(w_1,\ldots,w_e)\in \N^e$ such that 
$h -n_i - n_j = \sum_{\ell=1}^e w_\ell n_\ell$.
By the assumptions on $i,j$, we have
\[
F(\bv) = 
\bx^{\bv^+}-\bx^{\bv^-}
=
x_i \left( \frac{\bx^{\bv^+}}{x_i} - x_j \bx^\bw\right) 
- x_j \left(\frac{\bx^{\bv^-}}{x_j}-x_i\bx^{\bw}\right),
\] 
where $\bw=(w_1,\ldots,w_e)$.
Since both binomials in parentheses belong to $I_H$, this contradicts  the fact that $F(\bv)$ is a minimal binomial.

Since $h - n_i - n_j \notin H$, there exists $f \in \PF(H)$ so that 
$h - n_i - n_j \leq_H f$.
In particular, there exists ${\bf z} = (z_1,\ldots,z_e)$ such that
 $h' = \sum_{\ell=1}^e z_\ell n_\ell\in H$ satisfies
$f = (h - n_i - n_j) + h'$.
Observe that we must have $z_i=z_j =0$,
otherwise $f \in H$.
Then, 
the vectors $\bm_i=\bv^- - \be_i - \be_j + {\bf z}$ and $\bm_j=\bv^+ -\be_i-\be_j+{\bf z}$ are elements of $\Z^e$ such that $\sigma_i(\bm_i)=-1$, $\sigma_j(\bm_j)=-1$, 
$\sigma_\ell(\bm_i) \ge 0$ for all $\ell \ne i$,  $\sigma_\ell(\bm_j) \ge 0$ for every $\ell \ne j$, and $\sum_{\ell=1}^e \sigma_\ell(\bm_i)n_\ell = \sum_{\ell=1}^e \sigma_\ell(\bm_j)n_\ell = f$.
Thus, $\bm_i$ and $\bm_j$ may be chosen as the $i$-th and $j$-th row of an $\RF$-matrix $\RF(f)$.
For this matrix, we have  $F_{ji}(\RF(f)) = F(\bm_j-\bm_i)=F(\bv)$, concluding the proof.
\end{proof}

As a consequence, we have affirmatively answered 
 \cite[Question 3.6]{HW}.

\begin{cor}
For any numerical semigroup $H$,
the defining ideal  $I_H$ is generated by $\RF$-relations.
\end{cor}

\section{Critical toric ideals}\label{SectionCriticalIdeals}

Let $H$ 
be a numerical semigroup minimally generated by $\{n_1, \ldots, n_e\}$.

\begin{defn}[\cite{El83}]\label{DefCriticalToric}
The toric ideal $I_H$ is {\em critical} if it is generated by critical binomials, one for each variable $x_1, \ldots, x_e$.
\end{defn}

The goal of this section is to investigate numerical semigroups whose toric ideal is critical, specifically, to describe the arithmetic constraints forced on $\ker(\Phi_H)$ by this condition.
For some results on this topic, see \cite{AV,KO,GLO}.

Observe that if $I_H$ is critical then $\mu(I_H) \leq e$, thus, it is either a complete intersection or  an almost complete intersection.
The converse is not true in general (see Examples \ref{symnotcrit} \ref{acinotcrit}), however, see Theorem \ref{ThmCriticalIffAci} for a partial converse.
We also remark that being critical is a stronger condition than just being generated by critical binomials (see Example \ref{genbycritnotcrit}).

\begin{note}\label{NotationCritical}
Assume $I_H$ is critical. For each $i= 1, \ldots, e$, we fix a critical binomial
$$ 
F({\bf v}_i)=
x_i^{\al_i} - \prod_{j=1, j\ne i}^e x_j^{b_j} \in I_H
$$
where $b_j \in \N$ are such that $\al_in_i = \sum_{j=1, j\ne i}^e b_jn_j$,
and
 $\bv_i = \bv_i^+-\bv_i^- \in \Z^e$  where  ${\bf v}_i^+ = \al_i{\bf e}_i$ and 
${\bf v}_i^- = \sum_{j=1, j\ne i}^e b_j{\bf e}_j$.
\end{note}

We remark that, while the choice of $\bv_i$ is not unique in general, 
it will be unique for the semigroups that are the subject of this paper (see the proof of Theorem \ref{forwardmainTh}).

\begin{lem}\label{LemChainSum}
Assume that $I_H$ is critical, and adopt Notation \ref{NotationCritical}.
If ${\bf 0} \ne\bw \in \ker(\Phi_H)$ is  such that $\sigma_k(\bw^+) < \alpha_k$ for every $k=1,\ldots,e$, 
then there exists a chain 
$$\bw^+ = {\bf z}_0 \rightarrow {\bf z}_1 \rightarrow \cdots \rightarrow {\bf z}_r = \bw^-$$
such that, for every $j=1,\ldots,r$, we have ${\bf z}_j \in \N^e$ and ${\bf z}_j - {\bf z}_{j-1}=  \bv_{i_j}$, for some $i_j$. 
\end{lem}

\begin{proof}
Since $\bw \in \ker(\Phi_H)$, it follows that $\bw^+, \bw^-$ are factorizations of the same element $h \in H$. 
Since $I_H$ is generated by the critical binomials $F(\bv_1),\ldots, F(\bv_e)$,
it follows that
there exists a chain of factorizations of $h$ of minimal length
\begin{equation}\label{EqChain}
\bw^+ = {\bf z}_0 \rightarrow {\bf z}_1 \rightarrow \cdots \rightarrow {\bf z}_r = \bw^-
    \end{equation}
such that, for every $j=1,\ldots,r$, we have ${\bf z}_j \in \N^e$, and ${\bf z}_j - {\bf z}_{j-1}=\pm {\bf v}_{i_j}$, for some $i_j$. 
Thus, at each step we go from ${\bf z}_{j-1}$ to ${\bf z}_j$ by either adding or subtracting ${\bf v}_{i_j}$. 
The goal of this proof is to show that, actually,  \emph{subtractions} of vectors $\bv_i$ never occur in a minimal chain such as \eqref{EqChain}.

Since $\sigma_k({\bf z}_0) < \alpha_k$ for every $k$, at the first step we must have 
${\bf z}_1  ={\bf z}_{0}+ {\bf v}_{i_1}$.
Assume by contradiction that 
there exists an index $j > 1$ such that ${\bf z}_j -{\bf z}_{j-1}=-{\bf v}_{i_j}$, and choose the smallest such $j$.
Thus, we have
$$
{\bf z}_j = \bw^+ + {\bf v}_{i_1} + \cdots + {\bf v}_{i_{j-1}}-{\bf v}_{i_j},
$$
and it follows that
$$
0 \le \sigma_{i_j}({\bf z}_j) = \sigma_{i_j}(\bw^+) + \sigma_{i_j}({\bf v}_{i_1}) + \cdots + \sigma_{i_j}({\bf v}_{i_{j-1}})-\alpha_{i_j}.
$$
Since $\sigma_{i_j}(\bw^+) < \alpha_{i_j}$, 
we conclude that $ \sigma_{i_j}({\bf v}_{i_1}) + \cdots + \sigma_{i_j}({\bf v}_{i_{j-1}}) > 0$, 
which implies that $\sigma_{i_j}({\bf v}_{i_\ell}) > 0$ for some index $\ell<j$. 
Choose the largest such $\ell$. 
By the structure of the vectors $\bv_i$, we have
 $\sigma_{i_j}({\bf v}_{i_\ell}) > 0$ if and only if $i_j = i_\ell$, 
 and maximality of $\ell$ implies that $\sigma_{i_j}({\bf v}_{i_s}) \le 0$ for every $\ell < s < j$. 
 Then, since $\sigma_{i_j}(\bw^+) + \sigma_{i_j}({\bf v}_{i_1}) + \cdots + \sigma_{i_j}({\bf v}_{i_{j-1}}) \ge \alpha_{i_j}$, we must have 
 \begin{equation}\label{EqSigmaij}
 \sigma_{i_j}({\bf z}_s) =
      \sigma_{i_j}(\bw^+) + \sigma_{i_j}({\bf v}_{i_1}) + \cdots + \sigma_{i_j}({\bf v}_{i_{s}}) \ge \alpha_{i_j} \quad \text{ for every }\ell < s < j.
      \end{equation}
 With these choices, we argue that we can consider a shorter chain than \eqref{EqChain}, obtained by skipping the steps $\ell$ and $j$: namely, the chain
\begin{equation}\label{EqShorterChain}
\bw^+ = {\bf z}_0 \rightarrow {\bf z}_1 \rightarrow \cdots \rightarrow {\bf z}_{\ell-1} \rightarrow {\bf z}'_{\ell+1}  \rightarrow 
\cdots \rightarrow {\bf z}'_{j-1} \rightarrow {\bf z}_{j+1}  \rightarrow \cdots \rightarrow
{\bf z}_r = \bw^-
    \end{equation}
 where ${\bf z}'_{s}  = {\bf z}_{s} - {\bf v}_{i_\ell}$ for all $ \ell < s < j$. 
 In order to show that \eqref{EqShorterChain} contradicts the minimality of \eqref{EqChain}, 
we must show that ${\bf z}'_{s} \in \N^e$ for each $s$, and that each step is obtained by adding or subtracting one $\bv_i$.
The latter condition is clear, by construction of the chain:
\begin{align*}
{\bf z}'_{\ell+1}-{\bf z}_{\ell-1}&={\bf z}_{\ell+1}-{\bf v}_{i_\ell}-{\bf z}_{\ell-1}= \bv_{i_{\ell+1}},
\\
{\bf z}'_{s}-{\bf z}'_{s-1}&=
{\bf z}_{s}-{\bf z}_{s-1}  = \bv_{i_s}  & \text{ if }\ell < s-1, s < j,\\
{\bf z}_{j+1}-{\bf z}'_{j-1}&= {\bf z}_{j+1}-{\bf z}_{j-1}+{\bf v}_{i_\ell} =
\pm {\bf v}_{i_{j+1}} - {\bf v}_{i_j}+{\bf v}_{i_\ell} = \pm {\bf v}_{i_{j+1}},
\end{align*}
 where we used that $i_\ell = i_j$ in the last equation.

We now verify the former condition, that ${\bf z}'_{s} \in \N^e$ for each $\ell <s<j$.
If $k \ne i_j$, then 
$$
\sigma_k({\bf z}'_{s}) = \sigma_k({\bf z}_{s})-\sigma_k({\bf v}_{i_\ell})
= \sigma_k({\bf z}_{s})+\sigma_k({\bf v}^-_{i_\ell}) \geq  \sigma_k({\bf z}_{s}) \geq 0,
$$ 
 whereas for $k = i_j$ we have
$$
\sigma_{i_j}({\bf z}'_{s}) = \sigma_{i_j}({\bf z}_{s})-\sigma_{i_j}({\bf v}_{i_\ell})
= \sigma_{i_j}({\bf z}_{s})-\sigma_{i_j}({\bf v}^+_{i_\ell}) 
= \sigma_{i_j}({\bf z}_{s})-\alpha_{i_\ell}
$$ 
and this entry is non-negative by \eqref{EqSigmaij}.
This concludes the proof.
\end{proof}

\begin{cor}\label{ideal membership}
Assume that  $I_H$ is critical.
For every ${\bf 0}\ne\bw \in \ker(\Phi_H)$ there exists an index $k$ such that  $\sigma_k(\bw^+) \ge \alpha_k$ or $\sigma_k(\bw^-) \ge \alpha_k$. 
\end{cor}

\begin{proof}
Suppose that  $\sigma_k(\bw^+) < \alpha_k$ for every $k=1,\ldots,e$. 
Then, it follows by Lemma \ref{LemChainSum} 
that $\bw^- = \bw^+ + \sum_{j=1}^r\bv_{i_j}$.
Since $\bw \ne {\bf 0}$ we have $\bw^- \ne \bw^+$ and thus $r>0$.
It follows that $\sigma_{i_r}(\bw^-)=\sigma_{i_r}({\bf z}_{r-1})+\alpha_{i_r} \ge \alpha_{i_r}$, as desired.

If  $\sigma_k(\bw^-) < \alpha_k$ for every $k=1,\ldots,e$, apply the previous argument to $-\bw$.
\end{proof}

Next, we recall the notion of gluing of two numerical semigroups
$H_1$ and $H_2$.
Let $d_1 \in H_1$ and 
$d_2 \in H_2$ be non-zero elements that are not minimal generators and such that $\gcd(d_1,d_2) = 1$.
Then,  $H = d_2H_1 + d_1H_2$ is a numerical semigroup, 
  called a \emph{gluing} of $H_1$ and $H_2$.

\begin{lem}\label{LemmaGluing}
If there exists an index $i$ and a minimal set of binomial generators of $I_H$ such that $x_i$ appears only in one  generator,
then $H$ is a gluing.
\end{lem}
\begin{proof}
Assume without loss of generality that $i=1$.
The minimal generator of $I_H$ involving $x_1$ must be a critical binomial $ F (\bv_1)$ for $n_1$, encoding a relation
\begin{equation}\label{EqGluingSpecial}
\alpha_1 n_1 = b_2 n_2 + \cdots + b_e n_e,
\end{equation}
so that $\bv_1 = (\alpha_1, -b_2, \ldots, -b_e)$.
Denote $L = \ker(\Phi_H) = \ker(n_1, \ldots, n_e)\subseteq \Z^e$.
By assumption, there exist vectors $\bw_2, \ldots, \bw_t \in L$ 
such that $I_H$ is generated by $F_1$ and $F(\bw_2), \ldots, F(\bw_t)$ and $\sigma_1(\bw_j) = 0$ for all $j = 2,\ldots, t$.
Since $\bv_1,\bw_2, \ldots, \bw_t $ must generate the lattice $L$,
it follows that 
$$
\alpha_1 = \gcd\big\{\sigma_1(\bw)\,:\, \bw \in L\big\}.
$$
Since $\gcd(n_1, \ldots, n_e) = 1$, 
it follows from \eqref{EqGluingSpecial} that 
 $g = \gcd(n_2, \ldots, n_e)$ divides $\alpha_1$.
We claim that $g = \alpha_1$.
Let $\alpha_1 = g \alpha'_1$, $n_j = gn'_j$ for $j \geq 2$.
Consider the lattice 
$L' = \ker(n_1, n'_2, \ldots, n'_e)\subseteq \Z^e$.
Letting $\bv'_1 = (\alpha'_1,-b_2, \ldots, -b_e)$,
it follows  that $L'$ is generated by 
$\bv'_1,\bw_2, \ldots, \bw_t $,
in particular,  that
\begin{equation}\label{EqGcdPrime}
\alpha'_1 = \gcd\big\{\sigma_1(\bw)\,:\, \bw \in L'\big\}.
\end{equation}
Let $H' = \langle n'_2, \ldots, n'_e \rangle$.
This is a numerical semigroup since
 $\gcd(n'_2, \ldots, n'_e) =1$.
 It follows that $an_1 \in H $ for all sufficiently large $a$, and by \eqref{EqGcdPrime} we conclude that $\alpha'_1=1$, that is, $\alpha_1= \gcd(n_2, \ldots, n_e)$, as desired.

Dividing \eqref{EqGluingSpecial} by $\alpha_1$ we see  that $n_1 \in H'$.
Since $\gcd(n_1, \alpha_1) = \gcd(n_1, n_2,\ldots, n_e)$,
we have showed that  $H$ is a gluing of $\N$ and $H'$ using $\alpha_1 \in \N$
and $n_1 \in H'$.
Indeed, since $n_1, \ldots, n_e$ is a minimal generating set of $H$,
we must have
$\alpha_1 \notin\{0,1\}$ and $n_1 \notin \{0,n'_2, \ldots, n'_e\}$.
\end{proof}

\begin{rem}
After posting the first draft of this manuscript on arXiv, 
we learned that Lemma \ref{LemmaGluing} can also be deduced from \cite[Theorem 1.4]{R97}.
\end{rem}

\begin{prop}\label{PropMinimalRelationVi}
Assume that $I_H$ is critical and  $H$ is not a gluing. 
Adopt Notation \ref{NotationCritical}.
There exist   $\mu_1,\ldots,\mu_e \in \Z_{>0}$ such that  
$\mu_1{\bf v}_1+\cdots + \mu_e{\bf v}_e = {\bf 0}$.
\end{prop}

\begin{proof}
By Lemma \ref{LemmaGluing} it follows that, for every $i=1, \ldots, e$, there exists $j \ne i$ such that $i\in \supp({\bf v}^-_j)$.

Consider $\bw =  \sum_{i=1}^e {\bf v}_i$. 
If $\bw = {\bf 0}$,  the conclusion follows. 
Assume  that $\bw \ne {\bf 0}$, and 
write $\bw= \bw^+ - \bw^-$. 
Then, $\bw \in \ker(\Phi_H)$, and, by the previous paragraph, 
$\sigma_k(\bw^+)<\al_k$ for all $k = 1, \ldots, e$.
By Lemma \ref{LemChainSum}, 
we have
$\bw^- = \bw^+ +\sum_{j=1}^r\bv_{i_j}$ for some $i_j$,
that is,  $\sum_{i=1}^e\bv_{i}+\sum_{j=1}^r\bv_{i_j}={\bf 0}$, 
implying the desired conclusion.
\end{proof}

\begin{rem}\label{RemNotGluing}
The assumption that $H$ is not a gluing holds when 
$H$ is almost symmetric but not symmetric, 
see \cite[Theorems 6.3 and 6.7]{N}.
Also, if $I_H$ is a complete intersection and $H \ne \N$, then $H$ is a gluing, see \cite[Proposition 9]{Delorme}. 
\end{rem}

\begin{cor}\label{new3.2}
Assume that $I_H$ is critical and  $H$ is not a gluing. 
Adopt Notation \ref{NotationCritical}.
If $\lambda_1, \ldots, \lambda_e\in \N$ are not all zero but  at least one of them is zero,
then $\sum_{i=1}^e \lambda_i \bv_i \ne {\bf 0}$.
\end{cor}

\begin{proof}
Since $I_H = (F(\bv_1), \ldots, F(\bv_e))$, the vectors $\bv_1, \ldots, \bv_e$ generate the subgroup $\ker(\Phi_H)  \subseteq \Z^e$, which has rank $e-1$. 
It follows that the kernel of the homomorphism $\Psi : \Z^e \to \Z^e$ defined by $\Psi(\be_i) = \bv_i$ has rank 1.
By Proposition \ref{PropMinimalRelationVi}, this kernel is generated by a vector $(\mu_1, \ldots, \mu_e)^\intercal \in\Z^e$ with strictly positive entries.
The assumptions of this corollary guarantee that 
$(\lambda_1, \ldots, \lambda_e)^\intercal $ is not proportional to 
$(\mu_1, \ldots, \mu_e)^\intercal $,
and therefore is not in $\ker(\Psi)$.
\end{proof}

A crucial technical step towards the proof of Theorem \ref{forwardmainTh} is establishing the existence of distinguished RF-matrices for all $f$ and critical binomials for $H$, where all the relevant entries are bounded by the numbers $\alpha_i$ of Notation \ref{NotationCritical}. 
We do this in the next two lemmas.

\begin{lem}\label{mij<alj}
Assume that $I_H$ is critical and  $H$ is not a gluing. 
For each $f \in \PF(H)$,
there exists a unique RF-matrix
	$\RF(f)= (m_{ij})$ so that $m_{ij} < \al_j$ for every $i,j$.   
\end{lem}

\begin{proof} 
Start with an arbitrary RF-matrix  $\RF(f) =(m_{ij})$, and suppose that $m_{ij}\ge \al_j$ for some $i,j$.

Let ${\bf m}_i = (m_{i1}, m_{i2}, \ldots , m_{ie})$ be the $i$-th row of $\RF(f)$,
and recall that 
\begin{equation}\label{EqRiRF}
    \sum_{k\ne i}m_{ik}n_k-n_i =f.
\end{equation}
Recall, from 
Notation \ref{NotationCritical},
 the vector 
$$\bv_j = \bv_j^+-\bv_j^-  = \al_j{\bf e}_j-\sum_{k=1, k\ne j}^e b_k{\bf e}_k.$$
We must have that  $b_i=0$, otherwise adding the equation 
\[  \sum_{k=1, k\ne j}^e b_k n_k-\al_j n_j =0\]   
to equation \eqref{EqRiRF} would express  $f$ as a non-negative combination of $n_1, \ldots, n_e$, contradicting the fact that $f \notin H$.
It follows that the   $\bm'_i=\bm_i -{\bf v}_j = (m'_{i1}, \ldots , m'_{ie})$ satisfies the conditions for being the $i$-th row of an RF-matrix for $f$: we have $m'_{ii}=-1$, $m'_{ik}\geq 0$ for $k \ne i$, and $\sum_{k=1}^e m'_{ik}n_k = f$ (since $\bv_j \in \ker(\Phi_H)$).
Thus, we may define another RF-matrix $(m'_{ij})$ for $f$, 
obtained from $(m_{ij})$
replacing the $i$-th row  $\bm_i$ with 
$\bm'_i$.

If $m'_{ij} < \al_j$ for every $i,j$, the result follows, otherwise, we repeat this step.
We claim that this process must terminate in a finite number of steps.
Assume by contradiction that this is not the case.
There are only finitely many choices for the $i$-th row of $\RF(f)$
(since they correspond to factorizations of $f+n_i$),
thus, we must come back to 
the same row vector after a finite sequence of steps. 
This implies an equation of the form
\[ \bv_{j_1}+ \bv_{j_2}+ \cdots + \bv_{j_r} = {\bf 0},\]
where the indices $j_1,\ldots,j_r$ are all different from $i$, contradicting Corollary \ref{new3.2}.
This completes the proof of existence.

Finally, assume that there are two different RF-matrices $(m_{ij}), (m'_{ij})$ for $f$ satisfying $m_{ij}, m'_{ij} < \alpha_j$ for every $i,j$. 
Let $\bm_i$ and $\bm'_i$ denote the respective $i$-th rows, and choose $i$ so   that $\bm_i\ne \bm'_i$. 
It follows from the properties of RF-matrices that  $\bw =\bm_i-\bm'_i$ satisfies $\bw \in \ker(\Phi_H)$ and  $-\alpha_j < \sigma_j(\bw) < \alpha_j$ for every $j=1,\ldots,e$, contradicting Corollary \ref{ideal membership}.
\end{proof}

\begin{lem}\label{betas}
Assume that $I_H$ is critical and  $H$ is not a gluing. 
For each $i = 1, \ldots, e$ there exists a critical binomial $ F(\bv_i) = x_i^{\alpha_i}-\bx^{\bv_i^-}$ such that $\sigma_j(\bv_i^-) < \alpha_j$ for all $j \neq i$.
\end{lem}

\begin{proof}	
Fix an index $i$ and start with an arbitrary critical binomial 
$ F(\bv_i) = x_i^{\alpha_i}-\bx^{\bv_i^-}$ for $x_i$.
Suppose that $\sigma_{j_1}(\bv_i^-)  \ge \alpha_{j_1}$ for some  $j_1 \neq i$. 
Let $h = \alpha_i n_i = \sum_{k=1}^e \sigma_k(\bv_i^-)n_k$. 

Let ${\bf z}_0 = \bv_i^-$. 
Since $\sigma_{j_1}(\bv_i^-)  \ge \alpha_{j_1}$, the vector ${\bf z}_1= {\bf z}_0-\bv_{j_1} \in \N^e$ gives  another factorization of $h$, that is,
 $\alpha_in_i = \sum \sigma_k({\bf z}_1) n_k$.
 This implies  that 
$$
(\alpha_i-\sigma_i({\bf z}_1))n_i \in \langle n_1,\ldots,n_{i-1},n_{i+1},\ldots,n_e \rangle,$$
and, by minimality of $\alpha_i$, it forces $\sigma_i({\bf z}_1)=0.
$
	
If $\sigma_j({\bf z}_1) < \alpha_j$ for all $j \neq i$, the result follows, otherwise, we repeat this step.
We claim that this process must terminate in a finite number of steps.
Assume by contradiction that this is not the case.
At  each step $s > 1$ we obtain a factorization of $h$ associated to a vector ${\bf z}_s = {\bf z}_{s-1}-{\bf v}_{j_s}$ for some index $j_s \neq i$. 
Since there are finitely many different factorizations for $h$, 
there must exist  $p\ne q$ such that ${\bf z}_p={\bf z}_q$, that is, 
$${\bf 0} = {\bf z}_p - {\bf z}_q = \bv_{j_{p+1}}+\bv_{j_{p+2}}+\cdots+\bv_{j_s},$$
	and the index $i$ does not appear among the indices $j_{p+1},\ldots,j_s$ contradicting Corollary \ref{new3.2}.
\end{proof}

\begin{rem}
In fact, we will show in the proof of Theorem \ref{forwardmainTh}
that under the hypotheses of Lemma \ref{betas} there is a unique choice of $\bv_i$.
\end{rem}

\section{Proof of the forward direction of the main theorem}

As in the previous sections, 
let $H$ 
be a numerical semigroup minimally generated by $\{n_1, \ldots, n_e\}$.

We will need the following result, which is a reformulation of \cite[Theorem 3.3]{Et1}.

\begin{lem}[{\cite[Theorem 3.3]{Et1}\label{Eto}}]
Suppose that $I_H$ is minimally generated by binomials $F_1, \ldots, F_e$ where $F_1=x_1^{\alpha_1}-x_2^{\alpha_2}$. 
Let $d = \gcd(n_1,n_2)$, and let $H'=\langle d,n_3,\ldots,n_e \rangle$.  
Define a surjective ring homomorphism 
$\varphi:\Bbbk[x_1,\ldots,x_e] \rightarrow \Bbbk[y_1,\ldots,y_{e-1}]$ by
$$
\varphi(x_1)=y_1^{\alpha_2},\quad \varphi(x_2)=y_1^{\alpha_1},\quad \varphi(x_i)=y_{i-1}\,\text{ for }i=3,\ldots,e.
$$
Then, $H'$ has embedding dimension $e-1$ and $I_{H'}=\varphi(I_H)=(\varphi(F_2),\ldots,\varphi(F_e))$. 
\end{lem}

\begin{lem}\label{EqAperySets}
	Under the assumptions of Lemma \ref{Eto}, we have $$\Ap(H, n_1)=\{\ell n_2 +\omega \,: \, \ell = 0, \ldots, \alpha_2-1, \omega \in \Ap(H', d)\}.$$
\end{lem}

\begin{proof}
From $\alpha_1n_1=\alpha_2n_2$ and the minimality of $\alpha_1,\alpha_2$, we deduce that $\gcd(\alpha_1,\alpha_2)=1$, $n_1=\alpha_2d$ and $n_2=\alpha_1d$. Consider the set
$$	A=\{\ell n_2 +\omega \,: \, \ell = 0, \ldots, \alpha_2-1, \omega \in \Ap(H', d)\}.$$
We prove that $A=\Ap(H,n_1)$.  
If $\omega + \ell n_2 = \omega'+\ell' n_2$ for some $\omega,\omega' \in \Ap(H',d)$,  $0 \le \ell <\ell' \le \alpha_2-1$, then $\omega = \omega'+(\ell'-\ell)n_2$, contradicting $\omega \in \Ap(H',d)$. 
Thus, $|A|=\alpha_2d = n_1$.

Next, let $\omega + \ell n_2 \in A$. 
Write $\omega=\omega_3n_3+\cdots+\omega_en_e$ with $\omega_i \in \N$.
Assume by contradiction that $\omega + \ell n_2 \notin \Ap(H,n_1)$, then there exist $\lambda_1,\ldots,\lambda_e \in \N$, with $\lambda_1 > 0$ such that 
$$\omega+\ell n_2 = \lambda_1n_1+\cdots+\lambda_en_e.$$
This equation induces the equation in $H'$ 
$$\omega  +(\ell \alpha_1-\lambda_1\alpha_2-\lambda_2\alpha_1)d = \lambda_3n_3+\cdots+\lambda_en_e.$$
Indeed,
since $\omega \in \Ap(H',d)$, we must have $\ell \alpha_1 -\lambda_1 \alpha_2 - \lambda_2 \alpha_1 \ge 0$. 
Then, letting ${\bf z}=(0,\omega_3,\ldots,\omega_e)$ and ${\bf z}'=(0,\lambda_3,\ldots,\lambda_e)$, we obtain a binomial 
$$
F={\bf y}^{\bf z}y_1^{\ell \alpha_1-\lambda_1\alpha_2-\lambda_2\alpha_1}-{\bf y}^{\bf z'} \in I_{H'}.
$$
By Lemma \ref{Eto}  there exists a binomial $G \in I_H$ such that $\varphi(G)=F$. 
In particular,  that there exist $\mu_1,\mu_2 \in \N$ such that $\varphi(x_1^{\mu_1}x_2^{\mu_2})=y_1^{\ell \alpha_1-\lambda_1\alpha_2-\lambda_2\alpha_1}$,
thus,
$$\mu_1n_1+\mu_2n_2 = (\ell \alpha_1-\lambda_1\alpha_2-\lambda_2\alpha_1)d.$$
Dividing by $d$ we obtain  $\mu_1\alpha_2 + \mu_2 \alpha_1 = \ell \alpha_1-\lambda_1\alpha_2-\lambda_2\alpha_1,$
that is,
$$(\mu_1+\lambda_1)\alpha_2 = (\ell-\lambda_2-\mu_2)\alpha_1.$$
Since $\lambda_1 > 0$, then $(\mu_1+\lambda_1)\alpha_2 > 0$. 
However, since $\gcd(\alpha_1,\alpha_2)=1$, it follows that the smallest positive solution of the previous equality is obtained when $\mu_1+\lambda_1=\alpha_1$ and $\ell-\lambda_2-\mu_2 = \alpha_2$: then $\ell \ge \alpha_2$, contradicting the definition of $A$. 

We have shown that  $\omega + \ell n_2 \in \Ap(H,n_1)$, hence $A \subseteq \Ap(H,n_1)$. 
Since $|A|=|\Ap(H,n_1)|$, we conclude  that $A = \Ap(H,n_1)$.
\end{proof}

\begin{comment}
We record the following exchange-like property of critical ideals. 
\begin{prop}\label{PropExchangeProperty}
Assume that $I_H$ is an almost complete intersection and that $\alpha_i n_i \ne \alpha_j n_j$ for all $i \ne j$.
Then, any choice of $e$ critical binomials, one for each $x_1, \ldots, x_n$, generates $I_H$ minimally.
\end{prop}

\begin{prop}\label{PropExchangeProperty}
Assume that $I_H$ is critical and that $\alpha_i n_i \ne \alpha_j n_j$ for all $i \ne j$.
Then, any choice of $e$ critical binomials, one for each $x_1, \ldots, x_n$, generates $I_H$ minimally.
\end{prop}

\begin{prop}\label{diamond}
Let $H$ be an almost symmetric numerical semigroup that is not symmetric.
Then,  $I_H$ is critical  if and only if $H$ is an almost  complete intersection.
Moreover, if these equivalent conditions hold, then $\alpha_i n_i \ne \alpha_j n_j$ for all $i\ne j$.
\end{prop}

\begin{proof}
Let $F_1, \ldots, F_e$ be critical  binomials for $x_1, \ldots, x_e$.
Since they are minimal, their images in the graded vector space $I_H/\mm I_H$
are non-zero. 
The condition $\alpha_i n_i \ne \alpha_j n_j$ then guarantees that they all have different degrees, and therefore are necessarily linearly independent.
It follows by Nakayama's lemma that $F_1, \ldots, F_e$ are part of a minimal set of generators of $I_H$.
However, since $I_H$ is critical, it is an almost complete intersection, so 
$F_1, \ldots, F_e$ must generate $I_H$.
\end{proof}
\end{comment}

The following theorem is the first main result of this section.

\begin{thm}\label{ThmCriticalIffAci}
Let $H$ be an almost symmetric numerical semigroup that is not symmetric.
Then,  $I_H$ is critical  if and only if $H$ is an almost  complete intersection.
If these equivalent conditions hold, then $\alpha_i n_i \ne \alpha_j n_j$ for all $i\ne j$, and $I_H$ is generated minimally by any choice of $e$ critical binomials, one for each $x_1, \ldots, x_e$.
\end{thm}
\begin{proof}	
We have already  observed, after Definition~\ref{DefCriticalToric}, that being critical implies being a complete intersection or an  almost complete intersection.
However, complete intersection numerical semigroups are symmetric,  
so the forward direction of the first claim follows.

Assume that $I_H$ is an almost complete intersection.
We show first that  $\alpha_i n_i \ne \alpha_j n_j$ for all $i\ne j$. 
Without loss of generality, we prove this claim for $i=1, j=2$, and we assume $n_2 > n_1$.
Assume by contradiction that $\alpha_1 n_1 = \alpha_2 n_2$. 
Let $d=\gcd(n_1,n_2)$, then $\alpha_1n_1=\alpha_2n_2=\frac{n_1n_2}{d}$, and
let $H'= \langle d, n_3, \dots, n_e \rangle$.
It follows from Lemma \ref{EqAperySets} that
that 
\begin{equation}\label{EqAperyFlat}
    \Ap(H, n_1)=\{\ell n_2 +\omega \,: \, \ell = 0, \ldots, \alpha_2-1, \omega \in \Ap(H', d)\},
\end{equation}
and that the elements $\ell n_2 +\omega$  are all distinct.
It follows from \eqref{EqAperyFlat} that the maximal elements of $\Ap(H',d)$ with respect to $\leq_{H'}$ are in bijection with the maximal elements of $\Ap(H,n_1)$, the bijection being $\omega \mapsto (\alpha_2-1)n_2+\omega$.
This yields a bijection  $\eta: \PF(H')\to \PF(H)$,
given by 
\begin{equation}\label{EqBijectionPFFlat}
\eta(f')= f' + (\alpha_2-1)n_2 +d-n_1.
    \end{equation}
Obviously, $\eta$ preserves the usual order of the integers, thus, it must send the maximum of $\PF(H')$ to the maximum of $\PF(H)$,
that is, $\eta(\Fr(H'))=\Fr(H)$.

Since $H$ is almost symmetric but not symmetric,  there exist $f_1, f_2 \in \PF(H)$ such that $f_1+f_2 = \Fr(H) $.
Using \eqref{EqBijectionPFFlat}, this means that there exist $f'_1, f'_2 \in \PF(H')$
such that 
$$
f'_1+ (\alpha_2-1)n_2 +d-n_1 + f'_2+ (\alpha_2-1)n_2 +d-n_1 = \Fr(H')+(\alpha_2-1)n_2 +d-n_1,
$$
thus,
$$
\Fr(H')=f'_1 + f'_2+ (\alpha_2-1)n_2 +d-n_1 .
$$
Observe that $(\alpha_2-1)n_2 +d-n_1$ is a positive multiple of $d$, since $ \alpha_2 \geq 2$ and $n_2>n_1$, thus, $(\alpha_2-1)n_2 +d-n_1\in H'$.
From $f'_1 \in \PF(H')$ it follows that $f'_1 +  (\alpha_2-1)n_2 +d-n_1 \in H'$, and from $f'_2 \in \PF(H')$ we obtain the contradiction $\Fr(H') \in H'$.

We have proved that if $I_H$ is an almost complete intersection then 
  $\alpha_i n_i \ne \alpha_j n_j$ for all $i\ne j$. 
Now, we show that $I_H$ is critical, and that any choice of critical binomials generates $I_H$ minimally.

Let $F_1, \ldots, F_e$ be critical  binomials for $x_1, \ldots, x_e$.
Since they are minimal binomials, their images in the graded vector space $I_H/\mm I_H$
are non-zero. 
The condition $\alpha_i n_i \ne \alpha_j n_j$ then guarantees that they all have different degrees, and therefore are necessarily linearly independent.
It follows by Nakayama's lemma that $F_1, \ldots, F_e$ are part of a minimal set of generators of $I_H$.
However, since $I_H$ is an almost complete intersection, we conclude that 
$F_1, \ldots, F_e$ must generate $I_H$.
\end{proof}

We observe that the toric ideal of a complete intersection (and thus symmetric) numerical semigroup may fail to be critical, see Example \ref{symnotcrit}.

We also remark that if $I_H$ is critical but fails to satisfy 
 $\alpha_in_i \ne \alpha_jn_j$ for all $i,j$, there may be choices of critical binomials that fail to generate $I_H$,
see Example \ref{critnotanychoice}.

\begin{prop}\label{mj=alj-1} 
Assume 
that $H$  is almost symmetric but not symmetric and that $I_H$ is critical.
 Every $f \in \PF(H)\setminus\{ \Fr(H)\}$ has
 an RF-matrix   $\MM = \RF(f)= (m_{ij})$ such that, for every  $i\ne j$, 
either $m_{ij} =\al_j -1$ and  $m_{ji} =0$, or $m_{ij} =0$ and $m_{ji} =\al_i -1$. 
In particular,
 $\MM$ has exactly $(e^2 -e)/2$ positive entries.
 \end{prop}
 
\begin{proof}
Fix an $f \in \PF(H), f \ne \Fr(H)$.
By Lemma \ref{mij<alj},
there exists an RF-matrix   $\MM = \RF(f)= (m_{ij})$ 
 such that  $m_{ij}< \al_j$ for every $i\ne j$.
Let  ${\bf m}_i = (m_{i1},  \ldots , m_{ie})$ denote the 
$i$-th row of $\MM$. 
 Fix two arbitrary $i \ne j$, and 
consider $\bv = {\bf m}_i - {\bf m}_j\in \ker (\Phi_H)$.
By Corollary \ref{ideal membership}, there exists an index $k$ such that   $\sigma_k({\bf v}^+) \geq \alpha_k$ or $\sigma_k({\bf v}^-) \geq \alpha_k$.
If $\sigma_k({\bf v}^+) = m_{ik}-m_{jk} \geq \alpha_k$, since $m_{ik} < \alpha_k$,
then necessarily  $k = j$, so that $m_{jk}=-1$, and $m_{ij}= \alpha_{j}-1$, thus, $\sigma_j({\bf v}^+) = \alpha_j$. 
Similarly, if $\sigma_k({\bf v}^-) \geq \alpha_k$ then we obtain $m_{ji}=\alpha_i-1$. 
In conclusion, we have $m_{ij}=\alpha_j-1$ or $m_{ji}=\alpha_i-1$.

Since $H$ is almost-symmetric, $f'= \Fr(H)-f$ is another pseudo-Frobenius number of $H$ with $f' \ne \Fr(H)$.
Repeating the argument for $f'$, we obtain an RF-matrix $\MM' =\RF(f') = (m'_{ij})$ such that, for every $i\ne j$, at least one of $m'_{ij}, m'_{ji}$ is non-zero.
By \cite[Proposition 4]{Mo} or \cite[Lemma 2.1]{E}, we have 
$m_{ij}m'_{ji}=0$ whenever $i \ne j$.
It follows that we cannot have both $m_{ij}=\alpha_j-1$ and $m_{ji}=\alpha_i-1$.
Indeed,
if both $m_{ij}>0$ and $m_{ji}>0$, then $m_{ij}m'_{ji}=m_{ji}m'_{ij}=0$ gives $m'_{ji}=m'_{ij}=0$, contradicting the fact that at least one of these two entries is non-zero.
Hence $m_{ij}=0$ or $m_{ji}=0$.
Together with $m_{ij}=\alpha_j-1$ or $m_{ji}=\alpha_i-1$ proved above, this yields the assertion, and
 concludes the proof.
\end{proof}

Next, we prove a strengthened version of Theorem \ref{ThmRFRelations} in the  assumptions of this section.

\begin{prop}\label{RF(f)}
Assume 
that $H$  is almost symmetric but not symmetric and that $I_H$ is critical.
For every $f \in \PF(H)\setminus\{ \Fr(H)\}$,
consider an RF-matrix $\MM=\RF(f)$ as in Proposition \ref{mj=alj-1}.
Then, the ideal  $I_H$ is generated by RF-relations of $f$ arising from $\MM$.
\end{prop}

\begin{proof}
By Remark \ref{RemNotGluing}, Lemma \ref{betas} and Theorem \ref{ThmCriticalIffAci}, the  ideal
 $I_H$ is generated by critical binomials   $F(\bv_i) = x_i^{\alpha_i}-\prod_{j \ne i} x_j^{b_{ij}},$ such that  $b_{ij} < \alpha_j$ for every $j$.  
Following Notation \ref{NotationCritical},
denote
$$\bv_i = \bv_i^+-\bv_i^-  = \al_i{\bf e}_i-\sum_{j=1, j\ne i}^e b_{ij}{\bf e}_j
\quad \text{for }i = 1, \ldots, e.$$

Fix an $f \in \PF(H) \setminus \{\Fr(H)\}$, and let 
$\mathbb{M} = (m_{pq})$ be an RF-matrix for $f$ satisfying the conditions of 
Proposition \ref{mj=alj-1}, that is,  for every $p\ne q$, we have  either $m_{pq}=\alpha_q-1$ and $m_{qp}=0$ or $m_{pq}=0$ and $m_{qp}=\alpha_p-1$. Denote by $\bm_1,\ldots,\bm_e$ the rows of $\bbM$.

Fix an index $i$. 
We are going to  prove the  statement of the proposition by showing that $F(\bv_i)$ is an RF-relation of $f$ arising from $\MM$, 
specifically,
by showing that there exists an index $p$ such that $\bm_p-\bm_i=\bv_i$.

Let $J_i$ denote the set of indices $j$ such that $b_{ij} \ne 0$.
Let $h=\alpha_in_i=\sum b_{ij} n_j$ be the degree of $F(\bv_i)$, and, 
for every $j \in J_i$, let $h_j=h-n_i-n_j$, so that 
	\begin{equation}\label{fact}
		h_j=(\alpha_i-1)n_i - n_j = \left(\sum_{k \in J_i \setminus\{j\}} b_{ik}n_k \right) + (b_{ij}-1)n_j - n_i 
	\end{equation}

If there exists  $p \in J_i$ such that $m_{pi} \neq 0$, then $m_{pi} = \alpha_i-1$, thus, 
from the $p$-th row of $\mathbb{M}$ we obtain $h_p \le_H f$. 
On the other hand, if $m_{ji}=0$ for every $j \in J_i$, 
then  $m_{ij}=\alpha_j-1$ for every $j \in J_i$ by Proposition \ref{mj=alj-1}  and, since $b_{ij} \le \alpha_j-1$ for every $j \in J_i$, from the $i$-th row of 
$\mathbb{M}$ we obtain
	$$h-n_i \le_H \left(\sum_{j \in J_i} (\alpha_j-1)n_j \right) - n_i \le_H f.$$
In conclusion, we always have $h - n_i - n_p \le_H f$ for some index $p \in J_i$.

Let 
 ${\bf z}\in \N^e$ be a factorization of $f- h_p$
(observe that $f- h_p = f-(h-n_i)+n_p \in H$) and let 
$$\bm_i'=(m'_{i1},\ldots,m'_{ie})=\bv_i^-+{\bf z}-\be_i-\be_p.$$
Then  $\sigma_i(\bm_i')=-1$,  $\sigma_\ell(\bm_i')\geq \sigma_\ell(\bv_i^--\be_i-\be_p) \ge 0$ for $\ell \ne i$ since $p \in J_i$,
 and $\sum_{\ell=1}^e \sigma_\ell(\bm'_i)n_\ell = f$. 
Let $\bbM'$ be 
 the RF-matrix  of $f$ obtained from $\bbM$ by replacing $\bm_i$ with $\bm_i'$. 
Consider also  $f'' = \Fr(H)-f \in \PF(H)$, and let $\bbM''=(m''_{hk})$ be the RF-matrix of $f''$ obtained from Proposition \ref{mj=alj-1}. 

Let $\ell$ be  such that $\sigma_\ell(\bv_i^--\be_i-\be_p) > 0$. 
 Then, 
 $m'_{i\ell}=\sigma_\ell(\bm'_i) > 0$. 
By \cite[Proposition 4]{Mo} or \cite[Lemma 2.1]{E} applied to $\bbM'=\RF(f)$ and $\bbM''=\RF(\Fr(H)-f)$, it follows that
$m'_{i\ell}m''_{\ell i}=0$, that is, $m''_{\ell i}=0$. 
Applying Proposition \ref{mj=alj-1} to $\bbM''$,
it follows that $m''_{i\ell}=\alpha_\ell-1$.
Applying again  \cite[Proposition 4]{Mo} or \cite[Lemma 2.1]{E}  to $\bbM=\RF(f)$ and $\bbM''=\RF(\Fr(H)-f)$, we obtain $m_{\ell i}m''_{i\ell}=0$, that is $m_{\ell i}=0$, and finally, applying again  Proposition \ref{mj=alj-1} to $\bbM$ we obtain 
$$
\sigma_\ell(\bm_i)=m_{i\ell}=\alpha_{\ell}-1 \ge b_{i\ell}\ge\sigma_\ell(\bv_i^--\be_i-\be_p)$$
 for all $\ell$ such that $\sigma_\ell(\bv_i^--\be_i-\be_p)>0$.
It follows that ${\bf z}'=\bm_i-\bv_i^-+\be_i+\be_p \in \N^e$ is a factorization of $f-h+n_i+n_p\in H$ such that $\sigma_i({\bf z}')= 0$ and  $\sigma_\ell({\bf z}') < \alpha_\ell$ for all $\ell=1,\ldots,e$. 
We also have $\sigma_p({\bf z}')=0$, since
$\bv_i^+-\be_i-\be_p+{\bf z}'$ is also a factorization of $f \notin H$.

Finally, the vector $\bm'_p = \bv_i^+ - \be_i - \be_p + {\bf z}'$ is such that $\sigma_p(\bm'_p) = -1$, $0 \le \sigma_\ell(\bm'_p) \le \alpha_\ell-1$ for every $\ell \ne p$ and $\sum_{\ell=1}^e \sigma_\ell(\bm'_p)n_\ell=f$. However, since by Lemma \ref{mij<alj} the matrix $\bbM$ is the unique RF-matrix of $f$ satisfying $m_{ij} < \alpha_j$ for every $j \ne i$, it follows that $\bm'_p=\bm_p$, and then
\begin{equation*}
\bm_p-\bm_i=\bm'_p-\bm_i=(\bv_i^+ - \be_i - \be_p + {\bf z}')-(\bv_i^- - \be_i - \be_p + {\bf z}')=\bv_i.\qedhere    
\end{equation*}
\end{proof}

Proposition \ref{RF(f)} is false without the almost complete intersection assumption, see Example \ref{ce_RF(f)}.
Thus, Theorem \ref{ThmRFRelations} is the best result one can hope for in general.
\\

We are now ready to prove the forward direction of Theorem \ref{mainTh}.

\begin{thm}\label{forwardmainTh}
Let $H=\langle n_1, n_2, \ldots, n_e\rangle$ be a numerical semigroup
minimally generated by $e\ge 3$ elements. 
Assume that $H$ is almost symmetric  and $I_H$ is an almost complete intersection.
Then:
\begin{enumerate}
\item $H$ is  pseudo-symmetric;
\item $e=2e'+1$ is odd;
\item up to reordering the generators of $H$,
$I_H$ is generated by the binomials
\[
x_i^{\al_i} - x_{i+1}{x_{i+e'+1}^{\al_{i+e'+1}-1}} \qquad  i = 1,2 ,\ldots, e \]
where we identify $x_j = x_{j-e}$ and $\al_j =\al_{j-e}$ if $j> e$.
Moreover, this is the unique minimal set of binomial  generators of $I_H$; 
\item $f=\frac{\Fr(H)}{2}$ has an $\RF$-matrix
$\RF(f) = (m_{ij})_{1\le i, j \le e},$
where 
\[ m_{ij} = \left\{
\begin{array}{ll}
-1 & {\rm if\, } i=j, \\
0  &  {\rm if \,} i < j \le i+e', \\
\al_j-1 & {\rm otherwise}.
\end{array}
\right.
\]
Note that if $i > e' +1$ the condition $i < j \le i+e'$ reads $i<j \le e$ and $1\le j\le i-e'-1$.  
Moreover, this is the unique $\RF$-matrix of $f$.
\end{enumerate}
\end{thm}
\begin{proof}
Fix a $f \in \PF(H)$ such that $f \ne \Fr(H)$.
By Theorem \ref{ThmCriticalIffAci}, $I_H$ is generated minimally by $e$ critical binomials. 
By Proposition \ref{RF(f)},
they are RF-relations arising from
the RF-matrix $\bbM$ of Proposition \ref{mj=alj-1}. 
Let ${\bf m}_1,\ldots,{\bf m}_e$ be the rows of $\bbM$, and let $\bv_1, \ldots, \bv_e$ be the vectors corresponding to the minimal generators of $I_H$, 
as in Notation \ref{NotationCritical}.

For every $i$ there exist  $p_i,q_i$ such that $F_i = F({\bf m}_{p_i}-{\bf m}_{q_i})$, that is, ${\bf m}_{p_i}-{\bf m}_{q_i}={\bf v}_i$. 
Since the only positive coordinate of ${\bf v}_i$ is the $i$-th coordinate, it follows that $q_i=i$. 
Next, we prove that we can rearrange the generators $\{n_1,\ldots,n_e \}$ in a way such that $p_i=i+1$ for every $i$, that is, $F_i = F({\bf m}_{i+1}-{\bf m}_i)$, equivalently, ${\bf m}_{i+1}-{\bf m}_i = {\bf v}_i$.  
Observe that a permutation of the generators $n_1, \ldots, n_e$ induces the corresponding permutation on the coordinates of the vectors $\bv_1,\ldots,\bv_e$ and on both rows and columns of the $\RF$-matrix $\MM$.

Fix the first generator $n_1$.  
Choose the second generator $n_2$ so that  $F_1=F(\bm_2-\bm_1)$. 
Assume that we have chosen the  generators $n_1,\ldots,n_i$, for some $i < e$, 
and that the vectors $\bv_1,\ldots,\bv_i$ and $\bm_1,\ldots,\bm_i$ satisfy  the relation $F_j=F(\bv_j)=F(\bm_{j+1}-\bm_{j})$  for every $j < i$. 
By the previous paragraph, we have $F_i = F(\bv_i)=F(\bm_{p_i}-\bm_i)$. 
If $p_i < i$,
then $${\bf v}_i = {\bf m}_{p_i}-{\bf m}_i= ({\bf m}_{p_i}-{\bf m}_{p_i+1})+ (\cdots) + ({\bf m}_{i-1} - {\bf m}_i)= -{\bf v}_{p_i} - {\bf v}_{p_i+1} - \cdots - {\bf v}_{i-1}$$
and therefore ${\bf v}_{p_i}+\cdots+{\bf v}_i = 0$, 
which contradicts Corollary \ref{new3.2} since
 $p_i < i < e$. 
Thus, we have $p_i > i$ and, up to swapping $n_{i+1}$ with $n_{p_i}$, we may assume $p_i = i+1$ and thus  $F_i = F(\bm_{i+1}-\bm_i)$.

This algorithm allows us to choose the generators $n_1,\ldots,n_{e-1}$ so that  $F_i =F(\bm_{i+1}-\bm_i)$ for every $i=1\,\ldots,e-1$. 
Naturally, in this way we have also fixed the final generator $n_e$.
Consider the relation $F_e = F(\bv_e)=F(\bm_{p_e}-\bm_e)$.
We must have $p_e < e$, thus, as before, 
  we obtain
$${\bf v}_e ={\bf m}_{p_e}-{\bf m}_e=-{\bf v}_{p_e} - {\bf v}_{p_e+1} - \cdots - {\bf v}_{e-1}$$
and therefore 
${\bf v}_{p_e}+\cdots+{\bf v}_e = 0$, and again  we deduce $p_e=1$ by Corollary \ref{new3.2}. 
In particular, we have proved that 
$$
\bv_1 + \cdots + \bv_e = 0.
$$

For the rest of the proof,
recall the convention of Section \ref{SubsectionCircle}. 
For example, we identify an index $j$ with $j-e$ if  $e < j \leq 2e$,
we consider row $e$ to be directly above row 1, etc.

For each $i$,
the vector  ${\bf v}_i={\bf m}_{i+1}-{\bf m}_i$
 has a unique positive entry, namely, the $i$-th entry,
 and this entry is at least 2 since $\{n_1, \ldots, n_e\}$  is the minimal generating set of $H$.
This implies 
\begin{equation}\label{EqEndCircle}
m_{i+1,i} = \alpha_i + m_{i,i} = \alpha_i-1 \ge 1
\end{equation}
and
\begin{equation}\label{EqCircular}
 m_{i+1,j} \leq m_{i,j} \qquad \text{for  }j \ne i.
\end{equation}
In particular, for $j \ne i$ we have that $m_{i+1,j}>0$ implies $m_{i,j}>0$, and then $\supp(\bm_{i+1}) \subseteq \supp(\bm_i) \cup \{i\}$.
Thus, for each $j$,  the $j$-th column of $\bbM$  contains a $-1$ in position $j$, then a string of consecutive 0s above it, and then a string of positive entries above it.

We claim that 
 the number of positive entries in each row of $\bbM$ is constant.
To see this, let $\pi_i$ denote the number of positive entries in $\bm_i$. 
It follows from Theorem \ref{ThmCriticalIffAci} that 
 ${\bf v}_i^-$ has at least two positive entries.
 Then, there exist at least two indices $j$ for which $m_{i+1,j}<m_{i,j}$,
 in particular, this happens for at least one $j \ne i+1$. For such $j$, we have $m_{i+1,j}-m_{i,j} < 0$. However, by Proposition \ref{mj=alj-1}, all positive entries of $\MM$ are equal to $\alpha_j-1$: therefore, this must imply that $m_{i+1,j}=0$ and $m_{i,j} > 0$. Then $\supp(\bm_{i+1}) \ne \supp(\bm_i) \cup \{i\}$, and $\pi_{i+1}\leq \pi_i$ for all $i$. Since this is a circle of inequalities, it follows that $\pi_i$ is constant.
 
Since the total number of positive entries is ${e \choose 2}= e\frac{e-1}{2}$ by Proposition \ref{mj=alj-1}, it follows that each row has $\pi_i = e'=\frac{e-1}{2}$ positive entries, in particular, $e = 2e'+1$ is odd.
This, together with Proposition \ref{mj=alj-1} and the structure of each column of  $\bbM$ established earlier, 
determines the matrix $\bbM$ completely, proving item (4) in the statement of the theorem. 
The description of $F_i=F(\bv_i)$ in item (3) also follows.

Moreover, the equation $F_i = F({\bf m}_{i+1}-{\bf m}_i)$ gives the minimal generators, proving point (3).

For each $j=1,\ldots,e$, the $j$-th column of $\bbM$ has one entry $-1$ and exactly $e'$ entries equal to $\alpha_j-1$, with all other entries zero.
Summing the equations $\sum_{j=1}^e m_{ij}n_j=f$ over $i=1,\ldots,e$, we obtain
\[
ef=\sum_{j=1}^e\bigl(e'(\alpha_j-1)-1\bigr)n_j.
\]
The right-hand side depends only on $H$ and is unchanged by reordering the generators. Hence, there is  a unique $f \in \PF(H)$ with $f \ne \Fr(H)$, that is, $H$ has type two. Since $H$ is almost symmetric, it is pseudo-symmetric.

It remains to show that $I_H$ has a unique minimal set of binomial generators and that $\bbM$ is the unique RF-matrix of $f$.

For the first statement,
 let $F(\bw)$ be another minimal generator. 
By the exchange property of minimal sets of generators, we may assume, without loss of generality, that $\{F(\bw), F(\bv_2),\ldots,F(\bv_e)\}$ is another set of generators of $I_H$.
Since every set of generators must contain a critical binomial for each $x_i$, and since $\alpha_i n_i \ne \alpha_j n_j $ for all $i \ne j$ by Theorem \ref{ThmCriticalIffAci}, it follows that $F(\bw)$ must be a critical binomial for $x_1$, thus,  $\bw_1^+=\alpha_1\be_1$. 
If $\sigma_\ell(\bw_1^-) < \alpha_\ell$ for every $\ell$, then $\bv_1^--\bw_1^- \in \ker(\Phi_H)$ is such that $\sigma_\ell(\bv_1^--\bw_1^-) < \alpha_\ell$ for every $\ell=1,\ldots,e$,
 contradicting Corollary \ref{ideal membership}.
Suppose then that $\sigma_{j_1}(\bw^-_1) \geq \alpha_{j_1}$ for some index $j_1 \ne 1$. 
Then ${\bf z}_1=\bw_1^- -\bv_{j_1} \in \N^e$ is again a factorization of $\alpha_1n_1$. 
If $\sigma_\ell({\bf z}_1) < \alpha_\ell$ for every $\ell$, 
we reach again a contradiction; 
 if $\sigma_{j_2}({\bf z}_1) \geq \alpha_{j_2}$ for some index $j_2 \ne 1$, we can repeat this step. 
 Since the number of factorizations of $\alpha_1n_1$ is finite, if this algorithm does not halt then we eventually obtain the same factorization twice, implying $\bv_{j_r}+\cdots+\bv_{j_s} = {\bf 0}$ for some  $j_r,\ldots,j_s$, and since $\bv_1$ does not appear in this sum, we reach a contradiction by Corollary \ref{new3.2}.

 For the last statement,
 suppose  there is another RF-matrix $\bbM'$ for $f$, 
 and assume without loss of generality that its first row is $\bm_1'\ne\bm_1 =-\be_1+\sum_{i=e'+2}^{2e'+1} (\alpha_i-1)\be_i.$
Since $\bm_1+\be_1$ and $\bm'_1+\be_1$ are distinct factorizations of $f+n_1$ and $I_H=(F(\bv_1),\ldots,F(\bv_e))$, 
they must be connected by a chain of factorizations of $f+n_1$   such that every two consecutive elements of the chain  differ by $\pm \bv_i$, for some  $i$. 
Since $\bv_i = \alpha_i\be_i-\be_{i+1}- (\alpha_{i+e'+1}-1)\be_{i+e'+1}$,
it follows  that $\bm_1+\be_1 \pm \bv_i \notin \N^e$ for every $i$, and thus such a chain cannot exist. 
\end{proof}

\begin{rem}
Semigroups whose toric ideal has a unique binomial set of minimal generators, as in item (3) of Theorem \ref{forwardmainTh}, are  called \emph{uniquely presented}, see for example \cite{GO}.
\end{rem}

The description of the binomial generators in Theorem \ref{forwardmainTh} implies that $I_H$ is a ``Huneke--Ulrich almost complete intersection'' \cite{HU,Kustin95}:

\begin{cor}\label{CorollaryHU}
Assume the hypotheses of Theorem \ref{forwardmainTh}.
There exist an $e\times e$ alternating matrix $\mathbb{X}$ and a $1 \times e$ matrix $\mathbb{Y}$, with entries in  $\mm = (x_1, \ldots, x_e)$,
such that $I_H$ is generated by the entries of the product $\mathbb{Y} \mathbb{X}$.
In particular, $I_H$ is in the linkage class of a complete intersection.
\end{cor}
\begin{proof}
Let $\sigma$ be the permutation of $\{1, \ldots, n\}$ defined by $\sigma(i)=1+(i-1)(e'+1)$,
that is, the permutation associated to the oriented cycle $i \rightarrow i+e'+1$.
Define $\mathbb{Y}=(x_{\sigma(1)},x_{\sigma(2)},\ldots,x_{\sigma(e)})$ and $\mathbb{X}=(m_{ij})$   by  
$$
m_{ij}=0 \quad \text{if} \,j \notin \{i-1,i+1\},\qquad m_{i,i-1}=x_{\sigma(i-1)}^{\alpha_{\sigma(i-1)}-1}, \qquad m_{i,i+1}=-x_{\sigma(i)}^{\alpha_{\sigma(i)}-1}.$$
Observe that $\mathbb{X}$ is alternating.
Then, the $i$-th entry of $\mathbb{Y} \mathbb{X}$ is   $x_{\sigma(i-1)}x_{\sigma(i-1)}^{\alpha_{\sigma(i-1)}-1}-x_{\sigma(i+1)}x_{\sigma(i)}^{\alpha_{\sigma(i)}-1}$.
The last statement follows from this matrix structure; see \cite[p. 1301]{HU} or \cite[p. 373]{Kustin95}.
\end{proof}

\begin{ex}
The semigroup $H=\langle 15,11,17,19,16 \rangle$ is almost symmetric and an almost complete interesection. 
We have $\alpha_1=\alpha_3=\alpha_4=\alpha_5=2$ and $\alpha_2=3$.
The matrices of Corollary \ref{CorollaryHU} are
$$
\mathbb{Y}=(x_1,\,x_4,\,x_2,\,x_5,\,x_3),\qquad \mathbb{X} = \begin{pmatrix}
0 && x_1 && 0 && 0 && -x_3   \\
-x_1  && 0 && x_4 && 0 && 0    \\
0 && -x_4 && 0 && x_2^2 && 0  \\
0 && 0 && -x_2^2 && 0 && x_5 \\
x_3 && 0 && 0 && -x_5 && 0
\end{pmatrix},$$
and the toric ideal is generated by the entries of the product
$$
\mathbb{YX}=(x_3^2-x_1x_4,\,x_1^2-x_2x_4,\,
x_4^2-x_2^2x_5,\,x_2^3-x_5x_3,\,x_5^2-x_1x_3).$$
\end{ex}

\section{Cascade polynomials}\label{SectionCascade}

In this section, we introduce a new class of matrices, called \emph{cascade matrices}, and investigate  their determinant and minors in detail.
To better capture its combinatorial structure, 
we define this class in a multivariate setting;
the matrices that play a prominent role in Theorems \ref{forwardmainTh} and \ref{conversemainTh}
are positive integer specializations of cascade matrices.

Let $e=2e'+1$ be a positive odd integer, and let $y_1, \ldots, y_e$ be variables.
The \emph{$e$-th cascade matrix}  $\YY_e$  is defined as
$$
(\YY_e)_{(i,j)}=\left\{\begin{array}{ll}
-1 & i=j \\
0  &  i < j \le i+e' \\
y_j & {\rm otherwise}.
\end{array} \right.
$$
For instance,  the ninth cascade matrix is
$$
\YY_9 =
\begin{pmatrix}
-1 & 0 & 0 & 0 & 0 & y_6 & y_7 & y_8 & y_9 \\
y_1 & -1 & 0 & 0 & 0 & 0 & y_7 & y_8 & y_9 \\
y_1 & y_2 & -1 & 0 & 0 & 0 & 0 & y_8 & y_9 \\
y_1 & y_2 & y_3 & -1 & 0 & 0 & 0 & 0 & y_9 \\
y_1 & y_2 & y_3 & y_4 & -1 & 0 & 0 & 0 & 0 \\
0 & y_2 & y_3 & y_4 & y_5 & -1 & 0 & 0 & 0 \\
0 & 0 & y_3 & y_4 & y_5 & y_6 & -1 & 0 & 0 \\
0 & 0 & 0 & y_4 & y_5 & y_6 & y_7 & -1 & 0 \\
0 & 0 & 0 & 0 & y_5 & y_6 & y_7 & y_8 & -1
\end{pmatrix}$$
The determinant of $\YY_e$ is called the \emph{$e$-th cascade polynomial} and denoted
$$
P_e = \det(\YY_e) \in \Z[y_1, \ldots, y_e].
$$
Its monomials are all squarefree, and we denote them by 
$$\by_\cS = \prod_{i \in \cS} y_i
\quad \text{where } \cS \subseteq \{1,\ldots,e\}.
$$
While  not obvious from the definition, 
we are going to show that all coefficients of $P_e$ are non-negative, except for the constant term, 
by providing an enumerative interpretation.

\begin{thm}\label{ThmCascadePolynomial}
Let $e = 2e'+1\geq 3$ be an odd  integer.
Let $\Gamma$ be the cycle with vertices $\{1,\ldots,e\}$ and edges $\{(i, i+e')\,:\, i = 1, \ldots, e\}$.  
Let $\cS \subseteq \{1,\ldots,e\}$.
Let $i_1,\ldots,i_r$ be the number of vertices of the connected components of 
the subgraph  of $\Gamma$  induced by  
 $\cC=\{1,\ldots,e\} \setminus \mathcal{S}$.

The monomial $\by_\cS$ appears in $P_e$ with coefficient $e'-\sum_{j=1}^r \left \lceil \frac{i_j}{2} \right \rceil$, and this coefficient is non-negative if $\cS$ is not empty.
In particular, $P_e$ has total degree $e$ and leading term  $e' \cdot y_1\ldots y_e$.
\end{thm}
\begin{proof}
Fix a non-empty subset $\cS \subseteq \{1,\ldots,e\}$, and let $c_\cS$ be the coefficient of  $\by_\cS$ in $P_e$. Consider
\begin{equation}\label{expsigma}
    P_e = \det \YY_e = \sum_{\sigma \in 
    \Sym(e)} \sgn(\sigma) \prod_{i=1}^e (\YY_e)_{i,\sigma(i)},
\end{equation}
where $\Sym$ denotes the symmetric group.
From the definition of $\YY_e$, we see that the terms of  \eqref{expsigma} with non-trivial contributions to $c_\cS$ correspond to permutations $\sigma \in \Sym(e)$ satisfying $\sigma(i) \not \in \{i,i+1,\ldots,i+e'\}$ for every $i \in \cS$, and $\sigma(i)=i$ for every $i \notin \cS$.
The contribution of such a permutation $\sigma$ to $P_e$ is equal to
 $$ \sgn(\sigma) \prod_{i=1}^e (\YY_e)_{i,\sigma(i)} = \sgn(\sigma) \cdot (-1)^{e-|\cS|} \by_\cS.$$
 
Consider the $e \times e$ matrix $\YY'_e = \YY_e|_{y_1, \ldots,y_e=1}+ I$, 
 obtained from $\YY_e$ by specializing  all the variables to 1 and adding 1 to every diagonal element (obtaining thus 0).
Let $\YY'_{e,\cS}$ be the submatrix of $\YY'_e$ whose rows and columns are indexed by $\cS$.  Since  $(\YY'_e)_{(i,j)} = 1$ if $j < i  \le j+e'$ and $(\YY'_e)_{(i,j)}=0$ otherwise, the permutations $\sigma \in \Sym(e)$ such that $\sigma(i) \notin \{i,i+1,\ldots,i+e'\}$ for every $i \in \cS$ and $\sigma(i)=i$ for every $i \notin \cS$ are exactly the permutations fixing the elements of $\cC$ which give  a non-zero contribution in the computation of $\det \YY'_e$. 
It follows that
\begin{equation}\label{minorcoeff}
 c_\cS = (-1)^{e-|\cS|} \det(\YY_{e,\cS}').
\end{equation}

Let $Q$ denote the  matrix associated to the  cyclic permutation $i\mapsto i+1$, that is, $Q_{i,j}=1$ if $i=j+1$ and $Q_{i,j}=0$ otherwise. 
It follows that  $$\YY'_e = Q + Q^2 + \cdots + Q^{e'}.$$ 
Denoting by ${\bf 1}_e$ the column vector where all entries are 1,
we have $\YY'_e(I+Q^{e'})=Q+\cdots+Q^{2e'} = J-I$, where $J={\bf 1}_e {\bf 1}_e^\intercal$ is the matrix whose entries are all equal to $1$. 
Moreover, since every row of $\YY'_e$ 
contains $e'$ non-zero entries, all equal to $1$, we have $\YY'_eJ=e'J$. 
We deduce that $$\YY'_e \left(\frac{1}{e'}J-I-Q^{e'} \right)=I,$$
in particular,  $\YY_e'$ is invertible, with inverse matrix  given by
\begin{equation}\label{inversemat}
(\YY'_e)^{-1}=\frac{1}{e'}J - I - Q^{e'}.
\end{equation}
The eigenvalues of $Q$ are $\omega^i$, where $\omega=\exp(\frac{2\pi\sqrt{-1}}{e}) \in \mathbb{C}$ and $i=0,\ldots,e-1$, with eigenvectors $(1,\omega^i,\omega^{2i},\ldots,\omega^{(e-1)i})^\intercal$. 
Then  $\YY'_e=Q+Q^2+\cdots+Q^{e'}$ has the same eigenvectors, with eigenvalues $\lambda_i = \sum_{j=1}^{e'} \omega^{ij}$, for $i=0,\ldots,e-1$. 
Clearly, $\lambda_0=e'$, and $$\lambda_i = \sum_{j=1}^{2e'} \omega^{ij} = \omega^i \left(\frac{\omega^{e'i}-1}{\omega^{i}-1}\right)$$ for  $i=1,\ldots,2e'$.
Since $\gcd(e',e)=1$, multiplication by $e'$ is bijective on $\Z/e\Z$, thus
$$\det \YY'_e = \prod_{i=0}^{2e'} \lambda_i = e' \prod_{i=1}^{2e'} \omega^i  \left(\frac{\omega^{e'i}-1}{\omega^{i}-1}\right)= e' \prod_{i=1}^{2e'} \omega^i  = e'.$$
In particular, if $\cS=\{1,\ldots,2e'+1\}$, then $\YY_{e,\cS}'=\YY'_e$ and $c_\cS=(-1)^{|\cS|-1} \det \YY_e' =  e'$.
This proves the last claim of the theorem.

In general, in order to compute $c_\cS$ using (\ref{minorcoeff}), we investigate $\det \YY'_{e,\cS}$. 
Let $\cC=\{1,\ldots,2e'+1\} \setminus \cS$; 
then, by Jacobi's Minor Theorem, we obtain
\begin{equation}\label{Jacobi}
\det(\YY_{e,\cS}')= \det(\YY_e') \cdot \det( (\YY_e')^{-1}_\cC)=e' \cdot \det( (\YY_e')^{-1}_\cC),
\end{equation}
where we denote by $(\YY_e')^{-1}_\cC$ the principal minor of the $e \times e$ inverse matrix $(\YY'_e)^{-1}$ whose rows and columns are indexed by $\cC$. 
Since $J={\bf 1}_e {\bf 1}_e^\intercal$, restricting to $\cC$ in (\ref{inversemat}) gives $$(\YY'^{-1}_e)_\cC= \frac{1}{e'}{\bf 1}_\cC {\bf 1}_\cC^\intercal + (-I-Q^{e'})_\cC.$$
We compute its determinant using the Matrix Determinant Lemma
$$\det((\YY'^{-1}_e)_\cC)=\det((-I-Q^{e'})_\cC + \frac{1}{e'}{\bf 1}_\cC {\bf 1}_\cC^\intercal)=(1+\frac{1}{e'}{\bf 1}_\cC^\intercal ((-I-Q^{e'})_\cC)^{-1}{\bf 1}_\cC )\det(-I-Q^{e'})_\cC,$$
 then, by (\ref{minorcoeff}) and (\ref{Jacobi}) it follows that
\begin{equation}\label{coeffinv1}
c_\cS = (-1)^{e-|\cS|} \det(-I-Q^{e'})_\cC \left( e'+ {\bf 1}_\cC^\intercal ((-I-Q^{e'})_\cC)^{-1}{\bf 1}_\cC  \right).
\end{equation}

We observe that, 
if $R$ is an $n \times n$ matrix, then ${\bf 1}_n ^\intercal R {\bf 1}_n$ equals the sum of all entries of $R$, 
in particular,  it is invariant with respect to any permutation of the indices, as is the determinant. 
We use this invariance to compute $c_\cS$ using  (\ref{coeffinv1}). However, in order to control the monomial in function of $\cS$, we also permute the labels of the rows and columns of the matrix. Our goal is to transform the matrix $(-I-Q^{e'})_\cC$ into a block diagonal matrix, in order to simplify the computation of both the determinant and the sum of its entries.

Since $\gcd(e,e')=1$, the set $\{\lambda e' + 1 \ | \ \lambda=0,\ldots,e-1\}$ is a complete system of residues modulo $e$ (seen in $\{1,\ldots,e\}$). Therefore, 
the rearrangement of the indices given by $\lambda e'+1 \mapsto \lambda$ is well-defined. 
By construction, this rearrangement sends pairs $\{i, i+e'\}$  to pairs $\{j,j+1\}$: therefore,  this rearrangement transforms $(-I-Q^{e'})$ into
$$-I-Q=\begin{pmatrix}
-1 & 0 & 0 & 0 &\cdots & -1 \\
-1 & -1 & 0 & 0 & \cdots & 0  \\
0 & -1 & -1 & 0 & \cdots & 0  \\
0 & 0 & -1 & -1 & \cdots & 0 \\
\vdots & \vdots & \vdots & \vdots & \vdots & \vdots \\
0 & 0 & 0 & 0 & \cdots & -1 \\
\end{pmatrix}.
$$
Observe that by construction the rows and columns of this matrix are labeled according to the cycle graph $\Gamma$, starting from the two indices $1$ and $e'+1$.

Since $\cS$ is not empty,  we may assume, 
up to a shift (which corresponds to shifting the labels along the cycle $\Gamma$),  that $(-I-Q)_\cC$ does not contain the first row of $(-I-Q)$. 
This implies that $(-I-Q)_\cC$ is a lower triangular block matrix, thus, $\det((-I-Q)_\cC)=(-1)^{e-|\cS|}$, and by substituting in (\ref{coeffinv1}) we obtain
\begin{equation}\label{coeffinv2}
c_\cS=e'+{\bf 1}_\cC^\intercal ((-I-Q)_\cC)^{-1}{\bf 1}_\cC.
\end{equation}
Moreover, $(-I-Q)_\cC$ is a block diagonal matrix with blocks 
$B_1,\ldots,B_r$ of the form 
$$
\begin{pmatrix}
-1 & 0 & 0 & 0 &\cdots & 0  \\
-1 & -1 & 0 & 0 & \cdots & 0  \\
0 & -1 & -1 & 0 & \cdots & 0  \\
0 & 0 & -1 & -1 & \cdots & 0 \\
\vdots & \vdots & \vdots & \vdots & \vdots & \vdots \\
0 & 0 & 0 & 0 & \cdots & -1 
\end{pmatrix}.$$

By construction, there is a bijection between the labels of the rows and columns of each block $B_i$ and the connected components of the subgraph of $\Gamma$ induced by $\cC$.
Then, the inverse matrix $((-I-Q)_\cC)^{-1}$ is again 
block diagonal, 
with blocks 
$B_1^{-1},\ldots,B_r^{-1}$.
 Thus, 
$${\bf 1}_\cC^\intercal ((-I-Q)_\cC)^{-1} {\bf 1}_\cC = \sum_{j=1}^r {\bf 1}_\cC^\intercal B_j^{-1} {\bf 1}_\cC.$$
Let $i_1,\ldots,i_r$ be the sizes of the blocks $B_1,\ldots,B_r$. By the previous paragraph, $i_1,\ldots,i_r$ correspond to the sizes of the connected components of the subgraph of $\Gamma$ induced by $\cC$.
We prove that ${\bf 1}_\cC^\intercal B_j^{-1} {\bf 1}_\cC = - \lceil \frac{i_j}{2} \rceil$.
Let $$N_j = B_j+I=\begin{pmatrix}
0 & 0 & 0 & 0 &\cdots & 0  \\
-1 & 0 & 0 & 0 & \cdots & 0  \\
0 & -1 & 0 & 0 & \cdots & 0  \\
0 & 0 & -1 & 0 & \cdots & 0 \\
\vdots & \vdots & \vdots & \vdots & \vdots & \vdots \\
0 & 0 & 0 & 0 & \cdots & 0 \\
\end{pmatrix}.$$ 
Then, $N_j$ is nilpotent with order $i_j$, and 
$$B_j^{-1}=-(I-N_j)^{-1}=-I-N_j-N_j^2-\ldots-N_j^{i_j-1},$$ 
in particular,
\[ (B_j^{-1})_{pq}=\left\{
\begin{array}{ll}
0  & (p < q)  \\
(-1)^{p+q+1} & ({\rm otherwise}).
\end{array}
\right.\]
It follows that the  sum of entries on the $q$-th column of $B_j^{-1}$ is $-1$ if $i_j-q$ is even, and $0$ if $i_j-q$ is odd, so the sum of all entries is $-\lceil \frac{i_j}{2} \rceil$.
From (\ref{coeffinv2}) we obtain the desired claim
 $$c_\cS = e'- \sum_{j=1}^r \left \lceil \frac{i_j}{2} \right \rceil.$$

In order to conclude the proof, we show that this difference is non-negative.
Let $\tau_\cC := \sum_{j=1}^r \left \lceil \frac{i_j}{2} \right \rceil$. 
If $|\cC| = e-1$,  $\cC$ induces a path of size $2e'$ and $\tau_\cC =e'$.
Next, assume that $\tau_\cC \le e'$, let $p \in \cC$ and let $\cC'= \cC \setminus \{p\}$. 
Since $\cC$ is a union of disjoint paths, $p$ will belong to a certain path $\mathcal{P}$  of size $i_j$. 
By removing $\{p\}$ we split $\mathcal{P}$ in two paths $\mathcal{P}'$ and $\mathcal{P}''$ of size $\ell'$ and $\ell''$ such that $\ell'+\ell''=i_j-1$ (where one of the paths may be empty). Then $\lceil \frac{\ell'}{2}\rceil +  \lceil \frac{\ell''}{2}\rceil \le \lceil \frac{i_j}{2} \rceil$ and $\tau_{\cC'} \le \tau_\cC \le e'$, 
proving the desired conclusion.
\end{proof}

\begin{rem}
The combinatorial interpretation provided in Theorem \ref{ThmCascadePolynomial} 
gives some immediate information about  coefficients of the cascade polynomial:
\begin{enumerate}
\item If $|\cS|=1$, then $\cC$ is a path of size $2e'$, and $c_\cS=e'-e'=0$.
\item If $|\cS|=2$, then $\cC$ is either a path of size $2e'-1$, and thus $c_\cS=0$, or a union of two paths of size $j_1,j_2$, such that $j_1+j_2=2e'-1$, and it is easy to check that again $c_\cS=0$.
\item The minimum degree of a non-constant term is  $3$: if $\cS=\{1,e'+1,2e'+1\}$, then $\cC$ is a path of size $2e'-2$, and $c_\cS=e'-e'+1=1$. 
\item If $|\cS| \ge e'+2$, then $|\cC| \le e'-1$, and thus  $c_{\cS} > 0$. 
\item If $|\cS|=e-1$, then $\cC$ is a single vertex, and thus $c_\cS=e'-1$.
\item If $|\cS|=e$, then $\cC = \emptyset$ and $c_\cS=e'$. 
\item In general, the argument of the proof of Theorem \ref{ThmCascadePolynomial} shows that if $\cS \subseteq \cS'$ then $c_\cS \le c_{\cS'}$.
\end{enumerate}
\end{rem}

\begin{cor}\label{detpos}
Let $e= 2e'+1\geq 3$ be an odd integer,  $a_1, \ldots , a_e$ be positive integers and let  $\bbM= (m_{ij})$ be the matrix
\[ m_{ij} = \left\{
\begin{array}{ll}
-1 & \text{ if }i=j, \\
0  & \text{ if } i < j \le i+e',  \\
a_j & \text{otherwise}.
\end{array}
\right.\]
Then $\det(\bbM) > 0$, unless $e=3$ and $a_1=a_2=a_3= 1$.
\end{cor}

\begin{proof}

Since $\det(\bbM) = P_e(a_1,\ldots,a_e)$ by Theorem \ref{ThmCascadePolynomial} we deduce $\det(\bbM) \ge e' \cdot a_1\cdot \ldots \cdot a_e - 1 > 0$.
\end{proof}

\section{Proof of the converse direction of the main theorem}

This section is devoted to proving the following result.

\begin{thm}\label{conversemainTh}
Let $e= 2e'+1\geq 3$ be an odd integer,
let $\ba = (a_1, \ldots , a_e)$ be positive integers and let  $\bbM= (m_{ij})$ be the matrix
\[ m_{ij} = \left\{
\begin{array}{ll}
-1 & \text{ if }i=j, \\
0  & \text{ if } i < j \le i+e',  \\
a_j & \text{otherwise}.
\end{array}
\right.\]
Assume $\ba \ne (1,1,1)$. Then,
\begin{enumerate}
    \item $\det(\bbM)>0$.
\end{enumerate}
Consider the column vector
\[(n_1, \ldots , n_e)^\intercal = \det(\bbM) \cdot \bbM^{-1}{\bf 1}_e,\]
 where  ${\bf 1}_e$ denotes the column vector of length $e$, 
whose every entry is $1$.
Then,
\begin{enumerate}
\setcounter{enumi}{1}
\item $n_i > 0$ for every $i=1,\ldots,e$.
\end{enumerate}
Moreover, if in addition $\gcd(n_1, n_2, \ldots, n_e)=1$, then
\begin{enumerate}
\setcounter{enumi}{2}
\item $H = \langle n_1, n_2, \ldots, n_e \rangle$ has embedding dimension $e$,
\item $H$ is pseudo-symmetric with $\PF(H) = \{\det(\bbM),2\det(\bbM)\}$,
\item $I_H$ is critical, in particular, it is an almost complete intersection,
\item the exponents $\alpha_i$ in the critical binomials $F(\bv_i)$ are $\alpha_i = a_i +1$ 
for $i=1,2, \ldots, e$,
\item $\bbM$ is an RF-matrix for $\det(\bbM)$.
\end{enumerate}

\end{thm}
The first item was proved in Corollary \ref{detpos}. 
We are going to prove each of the remaining items separately.
Throughout this section,
we assume the notation and hypotheses of Theorem \ref{conversemainTh}.
\begin{prop}
We have $n_i>0$ for all $i$.
\end{prop}
\begin{proof}
By definition  we have $(n_1,\ldots,n_e)^\intercal \bbM = \det(\MM) \cdot {\bf 1}_e$. Let $R_1,\ldots,R_e$ be the rows of $\MM$. Then $(n_1,\ldots,n_e)^\intercal R_i = \det(\MM)$ for every $i=1,\ldots,e$, and by subtracting two consecutive rows we obtain $(n_1,\ldots,n_e)^\intercal (R_i-R_{i+1})=0$ for every $i=1,\ldots,e$. Thus,
 $$(a_i+1) n_i = n_{i+1}+ a_{i+e'+1} n_{i+e'+1}$$    for all $i$, and the three coefficients we see are all strictly positive.

For now, sign means one of ``non-negative'' or ``negative''. 
Consider the cycle with vertices $\{1,\ldots,e\}$ and edges given by the law $i \mapsto i+e'$: if $n_i$ and $n_{i+e'}$ always have different sign, consecutive vertices of this cycle will have different sign, which is impossible for an odd cycle. Then there exists an index $i$ such that $n_i$ and $n_{i+e'}$ have the same sign.

Next, we show that if $n_i$ and $n_{i+e'}$ have the same sign, then also $n_{i-1}$ and $n_{i+e'-1}$ will have the same sign. In fact, from the relation $$(a_{i-1}+1)n_{i-1} = n_i + a_{i+e'}n_{i+e'}$$ we deduce that $n_{i-1}$ has the same sign as $n_i$ and $n_{i+e'}$, and from the relation
$$a_{i+e'-1}n_{i+e'-1} = n_{i+e'}+a_{i-1}n_{i-1}$$
we obtain that $n_{i+e'-1}$ also has the same sign as $n_i$ and $n_{i+e'}$.

We have shown inductively that $n_1,\ldots,n_e$ all have the same sign. If there exists an index $i$ such that $n_i=0$, then from
$$0=(a_i+1)n_i=n_{i+1}+a_{i+e'+1}n_{i+e'+1}$$
we quickly deduce that they must all be zero, contradicting Corollary \ref{detpos}.
Finally, assume by contradiction that $n_1,\ldots,n_e$ are all strictly negative, and assume without loss of generality that $n_1$ is the largest of them. Then from the equality $(n_1,\ldots,n_e)^\intercal R_1 = \det(\MM) $ we deduce
$$\det(\bbM)=-n_1 + \sum_{j=e'+2}^e a_jn_j \le -n_{e'+2} + \sum_{j=e'+2}^e a_jn_j < 0,$$
contradicting Corollary \ref{detpos}.
\end{proof}

\begin{prop}\label{PropSignedMinors}
Let $R_1, \ldots, R_e$ denote the rows of $\MM$, and 
let $\MM'$ be the $(e-1)\times e$ matrix whose rows are 
$R_1-R_2, R_2-R_3, \ldots, R_{e-1}-R_e$.

Then, $n_{i}$ is equal, up to sign, to the minor obtained by deleting the $i$-th column from $\MM'$.
\end{prop}
\begin{proof}
Fix an index $i$. Since $(n_1, \ldots , n_e)^\intercal = \det(\bbM) \cdot \bbM^{-1}{\bf 1}_e$, $n_i$ is the sum of the entries in the $i$-th row of the adjugate matrix of $\MM$, that is, the sum of the entries in the $i$-th column of the cofactor matrix of $\MM$. 
Given an $e \times e$ matrix, denote by $\Delta_j(\cdot)$ the minor obtained by removing the $i$-th column and the $j$-th row. Then, up to sign, we have 
$$n_i = \Delta_1(\MM) - \Delta_2(\MM) +  \cdots - \Delta_{e-1}(\MM)+ \Delta_e(\MM).$$
For any $j=0,\ldots,e-1$, denote by $\MM''_j$ the square matrix whose rows are $R_k-R_{k+1}$ if $k \le j$, and $R_k$ otherwise (then $\MM''_0=\MM$). The matrices $\MM''_j$ are constructed inductively, 
where the matrix $\MM''_j$ is obtained from the matrix $\MM''_{j-1}$ by replacing the $j$-th row $R_j$ with $R_j-R_{j+1}$, 
thus the matrices $\MM''_j$ are all row-equivalent to $\MM$. We claim that for every $j=0,\ldots,e-1$
$$\Delta_{j+1}(\MM''_j)=\Delta_{j+1}(\MM)-\Delta_j(\MM)+\cdots+ (-1)^j\Delta_1(\MM).$$
Since $\Delta_e(\MM''_{e-1})$ will coincide with the minor of $\MM'$ obtained by deleting the $i$-th column, this claim will conclude the proof. 
The claim is trivial for $j=0$.
Assume then that the claim is true up to $j < e-1$. The matrix $\MM''_{j+1}$ is obtained from $\MM''_j$ by replacing the $j+1$-th row $R_j$ with $R_j-R_{j+1}$: since the determinant is multilinear we obtain
$$\Delta_{j+1}(\MM''_j)=\det 
\begin{pmatrix} 
R'_1-R'_2 \\
\vdots \\
R'_{j-1}-R'_j \\
R'_j-R'_{j+1} \\
R'_{j+2} \\
\vdots \\ 
R'_e
\end{pmatrix} = \det 
\begin{pmatrix} 
R'_1-R'_2 \\
\vdots \\
R'_{j-1}-R'_j \\
R'_j \\
R'_{j+2} \\
\vdots \\ 
R'_e
\end{pmatrix} - \Delta_j(\MM''_{j-1})$$
where $R'_1,\ldots,R'_e$ are obtained the rows of $\MM$ by removing the $i$-th entry.

However, the matrix $$\begin{pmatrix} 
R'_1-R'_2 \\
\vdots \\
R'_j-R'_{j+1} \\
R'_{j+2} \\
\vdots \\ 
R'_e
\end{pmatrix}$$
is row-equivalent to the $(e-1) \times (e-1)$ matrix obtained from $\MM$ by removing the $j+1$-th row and the $i$-th column, thus its determinant will be $\Delta_{j+1}(\MM)$, and
$$\Delta_{j+1}(\MM''_j)=\Delta_{j+1}(\MM)-\Delta_j(\MM''_{j-1})=\Delta_{j+1}(\MM)- (\Delta_j(\MM)-\Delta_{j-1}(\MM) + \ldots + (-1)^j \Delta_1(\MM)),$$
from which the claim follows.
\end{proof}

For the rest of this section, 
we denote by  $R_1,\ldots,R_e$  the rows of $\bbM$, and let $\bw_i := R_{i+1}-R_{i}$ for $i=1,\ldots,e$. 

\begin{cor}\label{latticegen}
Assume $\gcd(n_1,\ldots,n_e)=1$. Consider the numerical semigroup  $H= \langle n_1,\ldots, n_e \rangle$, and let $\Phi_H : \Z^e \to \Z$ be the map defined in Section \ref{SectionBasicConcepts}.
Then, the sublattice $\ker(\Phi_H)\subseteq \Z^e$ is generated by $\bw_1, \ldots, \bw_{e-1}$.
Moreover, the lattice of relations among $\bw_1,\ldots,\bw_e$ is generated by the relation $\bw_1+\cdots+\bw_e={\bf 0}$.
\end{cor}

\begin{proof}
It follows from Proposition \ref{PropSignedMinors}
 that the sublattice $L\subseteq \Z^e$ generated by $\bw_1, \ldots, \bw_{e-1}$ is saturated.
Observe that $L$ has rank $e-1$: this follows  since $\det(\MM) \ne 0$
and $\bw_1, \ldots, \bw_{e-1}$  are rows of a matrix obtained from $\MM$ by row operations. 
However, $\ker(\Phi_H)$ is also saturated of rank $e-1$, and
clearly we have $L \subseteq \ker(\Phi_H)$, thus, 
we conclude that $L = \ker(\Phi_H)$.

For the second statement, consider the linear map $\Theta : \Z^e \to \Z^e$ defined by $\Theta(\be_i) = \bw_i$. 
Since $\bw_1+\cdots+\bw_e={\bf 0}$, the image of $\Theta$ is $L$, and it follows that $\ker(\Theta)$ is a sublattice of rank 1. 
Since $(1,\ldots,1)^\intercal$ generates a saturated sublattice of $\ker(\Theta)$ of the same rank, the conclusion follows.
\end{proof}

The assumption that $\gcd(n_1,\ldots,n_e)=1$ is necessary: if $\gcd(n_1,\ldots,n_e) \ne 1$, Corollary \ref{latticegen} may fail, see Example \ref{gcd hyp} (2).

For the rest of this section,  we
assume that $\gcd(n_1,\ldots,n_e)=1$, and
consider the numerical semigroup $H$ generated by $n_1, \ldots, n_e$.
\begin{prop}\label{walpha}
    For every $i=1,\ldots,e$, we have $\alpha_i=a_i+1$.
\end{prop}
\begin{proof}
For each  $i=1,\ldots,e$ we have
$$\bw_i = (a_i+1)\be_i-\be_{i+1}-a_{i+e'+1}\be_{i+e'+1}.$$ 
Fix and index $i$, let $1 \le b_i \le a_i$ be an integer, and assume that  there exists a factorization $b_in_i = c_1n_1+\cdots+c_en_e$. 
This relation can be associated to a vector $\bw \in \ker(\Phi_H)$ such that $\bw^+=b_i\be_i$. 
Since $\bw \in \ker(\Phi_H)$ we can write $\bw = \lambda_1\bw_1+\cdots+\lambda_{e-1}\bw_{e-1}$, where $\lambda_j \in \Z$. 
Up to adding multiples of $\bw_1+\cdots+\bw_e = {\bf 0}$, we can write $\bw = \lambda_1\bw_1+\cdots+\lambda_e\bw_e$, where $\lambda_1,\ldots,\lambda_e$ are positive integers.

For every $j = 1,\ldots,e$, the $j$-th coordinate of $\bw$ is $$\sigma_j(\bw)=\lambda_j(a_j+1)-\lambda_{j-1}-\lambda_{j-e'-1}a_j = a_j(\lambda_j-\lambda_{j-e'-1})+(\lambda_j-\lambda_{j-1}).$$

Since $\bw^+=b_i\be_i$, there are two possibilities, depending on the index $j$:
\begin{enumerate}
\item If $j \ne i$, then $a_j(\lambda_j-\lambda_{j-e'-1})+(\lambda_j-\lambda_{j-1}) \le 0$, implying one of the following two statements:
\begin{enumerate}
    \item $\lambda_j < \lambda_{j-e'-1}$ or $\lambda_j < \lambda_{j-1}$,
    \item $\lambda_j=\lambda_{j-1}=\lambda_{j-e'-1}$.
\end{enumerate}
\item If $j = i$, then $a_i(\lambda_i-\lambda_{i-e'-1})+(\lambda_i-\lambda_{i-1}) = b_i \le a_i$, and thus $\lambda_i \le \lambda_{i-e'-1}$ or $\lambda_i \le \lambda_{i-1}$.
\end{enumerate}

Let $k$ be an index such that $\lambda_k$ is maximal. From (2), we can assume that $k \ne i$. 
Then, we must have $\lambda_k = \lambda_{k-1}=\lambda_{k-e'-1}$ from (1). Repeating this argument we obtain $\lambda_j = \lambda_{j-1}=\lambda_{j-e'-1}$ for every index $j$ (observe that when $j=i$, since $\lambda_i$ will be a maximum, the statement (2) still forces $\lambda_i=\lambda_{i-e'-1}$ or $\lambda_i=\lambda_{i-1}$). Then $\lambda_1=\cdots=\lambda_e$, and 
$$\bw = \lambda_1\bw_1+\cdots+\lambda_e\bw_e = \lambda_1(\bw_1+\cdots+\bw_e)={\bf 0},$$
contradicting $\bw^+=b_i\be_i$.
\end{proof}

\begin{cor}
The semigroup $H$ has embedding dimension $e$.
\end{cor}

\begin{proof}
It follows  since $\alpha_i = a_i+1 \ge 2$ for all $i$.
\end{proof}

\begin{prop}\label{lattice membership}
Let ${\bf 0} \ne \bw \in \ker(\Phi_H)$. Then:
\begin{enumerate}
\item There exist two indices $k_1,k_2$ such that $\sigma_{k_1}(\bw^+) \ge a_{k_1}$ and $\sigma_{k_2}(\bw^-) \ge a_{k_2}$.
\item There exists an index $k$ such that $\sigma_k(\bw^+) \ge a_k+1$ or $\sigma_k(\bw^-) \ge a_k+1$. 
\end{enumerate}

\end{prop}

\begin{proof}
By Corollary \ref{latticegen} the lattice $\ker(\Phi_H)$ is generated by $\bw_1,\ldots,\bw_{e-1}$. Since $\bw_1+\cdots+\bw_e={\bf 0}$ we can write
$$\bw=\lambda_1\bw_1+\cdots+\lambda_e\bw_e,$$
where $\lambda_1,\ldots,\lambda_e$ are positive integers.

Fix an index $j$. Then
$$\sigma_j(\bw)=a_j(\lambda_j-\lambda_{j-e'-1})+(\lambda_j-\lambda_{j-1}).$$

Let $k$ be such that $\lambda_k$ is maximum among all $\lambda_j$. If $\lambda_k=\lambda_{k-e'-1}$, we replace the index $k$ with $k-e'-1$, and so on. Since $\bw \ne 0$, not all the $\lambda_j$ are equal, and moreover, since $\gcd(e,e'+1)=1$, we do not go back to the same index before covering all of them: therefore, we find an index $k$ such that $\lambda_k$ is the maximum among all $\lambda_j$, and $\lambda_k \ne \lambda_{k-e'-1}$. Then consider the two equations
\begin{eqnarray}
\sigma_k(\bw) & = &a_k(\lambda_k-\lambda_{k-e'-1})+(\lambda_k-\lambda_{k-1}), \\
\sigma_{k-e'-1}(\bw) & = & a_{k-e'-1}(\lambda_{k-e'-1}-\lambda_{k-1})+(\lambda_{k-e'-1}-\lambda_{k-e'-2}).
\end{eqnarray}
From the first equation we already deduce the first statement, since $\lambda_k > \lambda_{k-e'-1}$ and $\lambda_k \ge \lambda_{k-1}$ (applying the same argument for $-\bw$ yields the second index).
Assume that our second statement is false, that is, for every index $j$ we have $\sigma_j(\bw^+) \le a_j$ and $\sigma_j(\bw^-)\le a_j$. Since $\lambda_k > \lambda_{k-e'-1}$, from the first equation we have $\lambda_k = \lambda_{k-1} > \lambda_{k-e'-1}$, and then in order to have $\sigma_{k-e'-1}(\bw) > -a_{k-e'-1}-1$ we must have $\lambda_{k-e'-2} \le \lambda_{k-e'-1} < \lambda_{k-1}$.

Repeating this argument we obtain a chain of equalities $\lambda_k=\lambda_{k-1}=\cdots=\lambda_{k-j}$ and a chain of inequalities $\lambda_k > \lambda_{k-e'-1} \ge \cdots \ge \lambda_{k-e'-j}$ for every $j$, which contradict each other.
\end{proof}

\begin{prop}\label{defidealw_i}
The defining ideal $I_H$ is generated by $F(\bw_1), \ldots, F(\bw_e)$. In particular, $I_H$ is critical and an almost complete intersection.
\end{prop}

\begin{proof}
Obviously we have $(F(\bw_1),\ldots,F(\bw_e)) \subseteq I_H$. 
Assume that $(F(\bw_1),\ldots,F(\bw_e)) \neq I_H$, and let $F(\bw) \in I_H \setminus (F(\bw_1),\ldots,F(\bw_e))$ be a binomial of minimal degree. 
Then $\bw = \bw^+-\bw^- \in \ker(\Phi_H)$, and since $F(\bw)$ has minimal degree, $\bw^+$ and $\bw^-$ have disjoint support.

We show that there exists a chain
$$\bw^+ = {\bf z}_0 \rightarrow {\bf z}_1 \rightarrow \cdots \rightarrow {\bf z}_r = \bw^-$$
such that, for every $j=1,\ldots,r$, we have ${\bf z}_j \in \N^e$ and ${\bf z}_j - {\bf z}_{j-1}= \pm \bw_{i_j}$, for some $i_j$. 
This will imply our statement.

We construct this chain algorithmically. 
Let ${\bf z}_0=\bw^+$, and ${\bf z'}_0 = \bw^-$. 

At each step, consider a pair of factorizations $({\bf z}_i,{\bf z'}_j)$, 
starting with $({\bf z}_0,{\bf z'}_0)$. Let ${\bf z}={\bf z}_i-{\bf z'}_j \in \ker(\Phi_H)$. 
If ${\bf z} \ne {\bf 0}$, by Proposition \ref{lattice membership} there exists an index $k$ such that $\sigma_k({\bf z}_i)=\sigma_k({\bf z}^+) \ge a_k+1$ or $\sigma_k({\bf z'}_j)=\sigma_k({\bf z}^-) \ge a_k+1$. 
There are two possibilities:
\begin{enumerate}
    \item If $\sigma_k({\bf z}_i) \ge a_k+1$, consider ${\bf z}_{i+1}={\bf z}_i - \bw_k \in \N^e$ and continue the algorithm with the pair $({\bf z}_{i+1},{\bf z'}_j)$

    \item If $\sigma_k({\bf z'}_j) \ge a_k+1$, consider ${\bf z'}_{j+1}={\bf z}'_j - \bw_k \in \N^e$ and continue the algorithm with the pair $({\bf z}_i,{\bf z'}_{j+1})$.
\end{enumerate}
We obtain two chains 
\begin{align*}
\bw^+ &= {\bf z}_0 \rightarrow {\bf z}_1 \rightarrow \cdots  ,\\
\bw^- &= {\bf z'}_0 \rightarrow {\bf z'}_1 \rightarrow \cdots ,
\end{align*}
such that for every $i,j$, ${\bf z}_i,{\bf z'}_j$ are factorizations of $h \in H$, ${\bf z}_i - {\bf z}_{i-1}= - \bw_{j_i}$ and ${\bf z'}_j - {\bf z'}_{j-1} =  - \bw_{i_j}$.

We prove that this algorithm must halt, that is, ${\bf z}_i = {\bf z'}_j$ for some $i,j$:
then we can join the two chains at ${\bf z}_i = {\bf z'}_j$, invert the second one, and obtain the desired chain.

Let $(p,q)$ be any pair of indices. Then $F({\bf z}_p-{\bf z}'_q) \in I_H$, and since ${\bf z}_p$ and ${\bf z}'_q$ are obtained from ${\bf z}_0=\bw^+$ and ${\bf z}'_0=\bw^-$ 
by applying swaps involving elements in $\{\bw_1,\ldots,\bw_e\}$, we have $\deg F({\bf z}_p-{\bf z}'_q) = \deg F(\bw)$ and $F({\bf z}_p-{\bf z'}_q)-F(\bw^+-\bw^-) \in (F(\bw_1),\ldots,F(\bw_e))$. 
If $({\bf z}_p,{\bf z}'_q)$ do not have disjoint support, the binomial $F({\bf z}_p-{\bf z'}_q)$ can be written as a combination of binomials $F'_1,\ldots,F'_r \in I_H$ of smaller degree. Since these binomials of $I_H$ will have degree smaller than $\deg F(\bw)$, we have $(F'_1,\ldots,F'_r) \subseteq (F(\bw_1),\ldots,F(\bw_e))$ by the minimality assumption: then $F({\bf z}_p-{\bf z}'_q) \in (F(\bw_1),\ldots,F(\bw_e))$, and thus $F(\bw^+-\bw^-) \in (F(\bw_1),\ldots,F(\bw_e))$, contradiction.

Then, for every pair of indices $(p,q)$, the vectors ${\bf z}_p$ and ${\bf z}'_q$ must have disjoint support. In particular, this implies that there must be an index $k$ not contained in the support of the elements ${\bf z}_0,{\bf z}_1,\ldots$, that is, we never subtract the vector $\bw_k$ in the chain ${\bf z}_0,{\bf z}_1, \ldots$.
Since $\bw_1+\cdots+\bw_e={\bf 0}$ is the only minimal relation among $\bw_1,\ldots,\bw_e$, we deduce that the elements of the first chain are all distinct factorizations of $h \in H$, and thus this chain is finite. With a similar argument, we also deduce that the chain ${\bf z}'_0,{\bf z}'_1,\ldots$ must be finite. Therefore, the algorithm must halt. 
\end{proof}

\begin{prop}\label{detnot}
 We have   $\det(\MM)\in \PF(H)$.
\end{prop}

\begin{proof}
We prove that $\det(\MM) \notin H$.
This will imply our statement, since if $\det(\bbM) \not \in H$, for every index $i$ we can deduce a factorization of $\det(\bbM)+n_i \in H$ from the $i$-th row of the matrix $\bbM$,
showing that $\det(\bbM) \in \PF(H)$.
Assume by contradiction that $\det(\MM) \in H$. 
Since $(n_1, \ldots , n_e)^\intercal R_1 = \det(\bbM)$, we deduce that
$$n_1 + \det(\MM) = \sum_{i=e'+2}^e a_in_i,$$
thus, there exist $b_1,\ldots,b_e \in \N$, with $b_1 > 0$, such that 
    $$\sum_{i=1}^e b_i n_i = \sum_{i=e'+1}^e a_in_i.$$
Applying Proposition \ref{defidealw_i}, there exists a chain 
$$(0,\ldots,0,a_{e'+2},\ldots,a_e)={\bf a}={\bf z}_0 \rightarrow {\bf z}_1 \rightarrow \cdots \rightarrow {\bf z}_r = {\bf b}=(b_1,\ldots,b_e).$$ such that ${\bf z}_i-{\bf z}_{i-1}=\pm \bw_{j_i}$ for every $i=1,\ldots,r$. 
In particular, ${\bf z}_1 = {\bf a} \pm \bw_{j_1} \in \N^e$, for some index $j_1$. 
Since $\bw_i = (a_i+1)\be_i-\be_{i+1}-a_{i+e'+1}\be_{i+e'+1}$ for every $i=1,\ldots,e$, we reach a contradiction, since ${\bf a} \pm \bw_i \notin \N^e$ for every $i=1,\ldots,e$. 
\end{proof}

\begin{prop}\label{2detnot}
We have $2\det(\bbM) \in \PF(H)$.
\end{prop}

\begin{proof}
We prove that $2\det(\bbM) \notin H$. This will imply our statement: if $2\det(\bbM) \notin H$, since $(n_1, \ldots , n_e)^\intercal R_i = \det(\bbM)$ for every $i=1,\ldots,e$, by adding the equations associated to $R_i$ and $R_{i+e'+1}$ we obtain
$$2\det(\bbM)= -n_i+(a_{i+e'+1}-1)n_{i+e'+1}+\sum_{j=1,j \ne i+e'+1, j \ne i}^e a_jn_j.$$
Then $2\det(\MM)+n_i \in H$ for every $i=1,\ldots,e$, and thus $2\det(\MM) \in \PF(H).$

Assume  by contradiction that $2\det(\bbM) \in H$. 
Then there exists a factorization ${\bf z}$ of $2\det(\bbM)+n_1 \in H$ such that $\sigma_1({\bf z}) > 0$. On the other hand, from  $(n_1, \ldots , n_e)^\intercal (R_1+R_{e'+2})= 2\det(\bbM)$ we deduce 
\begin{equation}\label{EqnRow2det}
2\det(\bbM)+n_1= (a_{e'+2}-1)n_{e'+2}+\sum_{j=2,j \ne e'+2}^e a_jn_j.
\end{equation}
Let $\bw \in \N^e$ be the factorization of $2\det(\MM)+n_1$ associated to equation \eqref{EqnRow2det}. Then $\bw-{\bf z} \in \ker(\Phi_H)$ is such that $\sigma_1((\bw-{\bf z})^+) < a_1+1 = \alpha_1$, and since $I_H=(F(\bw_1),\ldots,F(\bw_e))$ by Proposition \ref{defidealw_i}, by Lemma \ref{LemChainSum} there exists a chain
$$\cC : (\bw-{\bf z})^+ = {\bf z}_{0} \rightarrow {\bf z}_{1} \rightarrow \cdots \rightarrow {\bf z}_{r} = (\bw-{\bf z})^-$$
such that, for every $j=1,\ldots,r$, we have ${\bf z}_{j} \in \N^e$ and ${\bf z}_{j} - {\bf z}_{j-1}=  \bw_{i_j}$, for some $i_j$.

Assume that the chain $\cC$ has minimal length; 
in particular, this means that the vectors $\bw_1,\ldots,\bw_e$ do not appear all as difference of elements in the chain, 
or else we would be able to remove them and find a shorter chain, 
arguing as in the proof of Lemma \ref{LemChainSum}, using $\bw_1+\cdots+\bw_e={\bf 0}$. 

Let $\tilde{\bw}_1=\bw_1$, $j_1=1$, define recursively the indices $j_i = j_{i-1}+e'+1$, and let $\tilde{\bw}_i = \tilde{\bw}_{i-1} + \bw_{j_i}$ (then $\tilde{\bw}_1=\bw_1$, $\tilde{\bw}_2=\bw_1+\bw_{e'+2}$, and so on). 
In particular, $\tilde{\bw}_e = \bw_1+\cdots+\bw_e$. 
We prove that the sums $\tilde{\bw}_i$ are sub-sums of $\bw_{i_1}+\cdots+\bw_{i_r}$; this will imply that $\tilde{\bw}_e = \bw_1+\cdots+\bw_e$ is a sub-sum of $\bw_{i_1}+\cdots+\bw_{i_r}$, contradicting the minimality of $\cC$.

We prove this claim by induction. First, since $\sigma_1({\bf z}_r) = \sigma_1((\bw-{\bf z})^-) > 0$ and $\sigma_1((\bw-{\bf z})^+) = 0$, there exists an index $s$ such that $i_s=1$. 
Then $\tilde{\bw}_1$ is a sub-sum of $\bw_{i_1}+\cdots+\bw_{i_r}$. 
Consider the index $j_2=j_1+e'+1=e'+2$. Then Equation (\ref{EqnRow2det}) implies $\sigma_{j_2}((\bw- {\bf z})^+-\bw_1) \le a_{e'+2}-1 - a_{e'+2} < 0$, and, since $\sigma_{j_2}((\bw-{\bf z})^-) \ge 0$,
it follows that there must be an index $s$ such that $i_s=e'+2$, and $\tilde{\bw}_2=\bw_1+\bw_{e'+2}$ is a sub-sum of $\bw_{i_1}+\cdots+\bw_{i_r}$.

Assume now that $\tilde{\bw}_i$ is a sub-sum of $\bw_{i_1}+\cdots+\bw_{i_r}$ for some $i \ge 2$. 
Consider the equation
\begin{equation}\label{poschain}
\sigma_{j_{i+1}}((\bw-{\bf z})^+ + \bw_{i_1}+\cdots+\bw_{i_r}) = \sigma_{j_{i+1}}((\bw-{\bf z})^-) \ge 0.
\end{equation}
By definition of $\bw_k$, the only indices $k$ for which $\sigma_{j_{i+1}}(\bw_k)$ is non-zero are $j_{i+1},j_{i+1}-1,j_{i+1}-e'-1$. However, since $j_{i+1}-1=j_{i+1}-2e'-2=j_{i-1}$ and $j_{i+1}-e'-1=j_i$ then $\bw_{j_{i-1}}+\bw_{j_i}$ is a sub-sum of $\tilde{\bw}_i$, and thus a sub-sum of $\bw_{i_1}+\cdots+\bw_{i_r}$. But then, since $$\sigma_{j_{i+1}}((\bw-{\bf z})^++\bw_{j_{i-1}}+\bw_{j_i}) \le a_{j_{i+1}}-a_{j_{i+1}}-1 < 0,$$
then by Equation (\ref{poschain}) it follows that there exists an index $s$ such that $i_s=j_{i+1}$, and thus $\tilde{\bw}_{i+1}=\tilde{\bw}_i+\bw_{j_{i+1}}$ is a sub-sum of $\bw_{i_1}+\cdots+\bw_{i_r}$.
\end{proof}

\begin{lem}\label{factorapery}
For every index $i$ and every element $\omega \in \Ap(H,n_i)$, there exists a factorization
$$\omega=z_1n_1+\cdots+z_en_e,$$
where $z_j \le a_j$ for every $j=1,\ldots,e$. 
\end{lem}

\begin{proof}
Fix an index $i$, and let $\omega \in \Ap(H,n_i)$. Then $\omega=z_1n_1+\cdots+z_en_e$, where $z_j \in \N$ for every $j=1,\ldots,e$ and $z_i=0$. Let ${\bf z}=(z_1,\ldots,z_e)$, and assume that there exists $k$ such that $\sigma_k({\bf z})\ge \alpha_k = a_k+1$. 
Then ${\bf z}-\bw_k \in \N^e$ is another factorization of $\omega$; since $\omega \in \Ap(H,n_i)$, the $i$-th coordinate of this new factorization is again $0$. 
By repeating this argument, we subtract several $\bw_j$'s from our factorization, but  cannot subtract $\bw_i$ since the $i$-th coordinate is always zero. Since the number of factorizations is finite, if this algorithm does not halt we eventually reach the same factorization, contradicting the fact that the only minimal relation among the $\bw_i$ is $\bw_1+\cdots+\bw_e={\bf 0}$. 
\end{proof}

\begin{prop}\label{PropPFfromM}
The  semigroup $H$ is pseudo-symmetric with $\PF(H)=\{\det(\bbM),2\det(\bbM)\}$.
\end{prop}
\begin{proof}
Let $h \in \PF(H) \setminus \{2\det(\bbM)\}$. Then $h+n_1  \in  \Ap(H,n_1)$, and by Lemma \ref{factorapery} we can write $$h=\omega-n_1=z_1n_1+\cdots+z_en_e,$$
where $0 \le z_j \le a_j$ for every $j=2,\ldots,e$, and $z_1=-1$. 
Let ${\bf z}=(z_1,\ldots,z_e)$.

Since $h \in \PF(H) \setminus \{2\det(\bbM)\}$, we have $h \not\le_H 2\det(\bbM)$.
We are going to  prove that $z_i = 0$ and $z_{e'+i}=a_{e'+i}$ for every $i=2,\ldots,e'+1$. 
This will determine a unique element of $\PF(H) \setminus \{2\det(\MM)\}$, which must be $\det(\MM)$ by Proposition \ref{detnot},  thus showing that $H$ is pseudo-symmetric.

We argue by induction. First, from $(n_1, \ldots , n_e)^\intercal (R_1+R_{e'+2}) = 2\det(\bbM)$ we deduce that
$$2\det(\bbM)= -n_1+(a_{e'+2}-1)n_{e'+2}+\sum_{j=2,j \ne e'+2}^e a_jn_j.$$
Since $a_j \ge z_j$ for every $j=2,\ldots,e$, if $z_{e'+2} \le a_{e'+2}-1$ we would have $h \le_H 2\det(\bbM)$, a contradiction. 
Then $z_{e'+2}=a_{e'+2}$.
Moreover, if $z_2 \ge 1$, since $\bw_1=(a_1+1)\be_1-\be_2-a_{e'+2}\be_{e'+2}$ then ${\bf z}+\bw_1 \in \N^e$ would be another factorization of $h$, contradicting $h \not \in H$.
Then $z_2=0$.

Assume now that $z_j=0$ and $z_{e'+j}=a_{e'+j}$ for every $j=2,\ldots,i-1$. 
Then, from the equality $(n_1, \ldots , n_e)^\intercal (R_1+R_{e'+i}) = 2\det(\bbM)$ we obtain
\begin{equation}\label{Eqni2det}
2\det(\bbM)= -n_1+\sum_{j=i}^{e'+1} a_jn_j + \sum_{j=e'+2}^{e'+i-1}2a_jn_j+ (a_{e'+i}-1)n_{e'+i}+\sum_{j=e'+i+1}^{e} a_jn_j.
\end{equation}
If $z_{e'+i} \le a_{e'+i}-1$, since $z_j \le a_j$ for every $j=2,\ldots,e$, we would have $h \le_H 2\det(\bbM)$, contradiction.
On the other hand, if $z_i \ne 0$, then since $z_{e'+j}=a_{e'+j}$ for every $j=2,\ldots,i-1$ we can see that $${\bf z}-\sum_{j=1}^{i-1} \bw_j \in \N^e$$
is a factorization of $h$ in $H$, contradicting $h \notin H$. Then, $z_i=0$ and $z_{e'+i}=a_{e'+i}$.
\end{proof}

\begin{ex}\label{ExException111}
The case $\ba= (1,1,1) $ is to be considered a genuine exception, in which the construction ceases to be meaningful.
In this case, the matrix
$$
\bbM = \begin{pmatrix}
    -1 & 0 & 1 \\ 1 & -1 & 0 \\ 0 & 1 &-1
\end{pmatrix}
$$
is singular, and the formula $(n_1, n_2, n_3) = \det(\bbM) \cdot \bbM^{-1}{\bf 1}_e $ cannot be applied directly. 
We may opt to define $n_1, n_2, n_3$ directly in terms of the adjugate matrix; however, we have
$$
\mathrm{adj}(\bbM) = \begin{pmatrix}
    1 & 1 & 1 \\ 1 & 1 & 1 \\ 1 & 1 & 1
\end{pmatrix},
$$
so we would obtain $n_1=n_2=n_3=3$ and the construction would not yield a numerical semigroup.
\end{ex}

\appendix

\section{Examples}
    This appendix provides examples and counterexamples demonstrating the  sharpness of some of the statements and results appearing in the previous sections.

\begin{ex} \label{gcd hyp}
\begin{enumerate}
\item Let ${\bf a}=(a_1,\ldots,a_5)=(2,3,4,5,6) \in \mathbb{Z}^5$ and let $\mathbb{M}$ be the matrix of Theorem \ref{mainTh}
\[ \mathbb{M}=\begin{pmatrix}
-1 & 0 & 0 & 5 & 6 \\
2 & -1 & 0 & 0 & 6 \\
2 & 3 & -1 & 0 & 0 \\
0 & 3 & 4 & -1 & 0 \\
0 & 0 & 4 & 5 & -1 
\end{pmatrix}.
\]
Then
\[ \det(\bbM) \cdot \bbM^{-1}{\bf 1}_e = (731,553,358,328,309)^\intercal\]
Since $\gcd(731,553,358,328,309)=1$, the numerical semigroup $H = \langle 731,553,358,328,309 \rangle$ is pseudo-symmetric and $\Bbbk[H]$ is an almost complete intersection, and we have $I_H= (x_1^3-x_2x_4^5,x_2^4-x_3x_5^6,x_3^5-x_4x_1^2,x_4^6-x_5x_2^3, x_5^7-x_1x_3^4)$.
\item Let ${\bf a}=(5,3,3,2,2)$. Repeating the same construction we obtain \[ \det(\bbM) \cdot \bbM^{-1}{\bf 1}_e =  (105,160,190,235,225)^\intercal.\]
Since $\gcd(105,160,190,235,225)=5$, the main Theorem does not apply here. However, we might consider the monoid $H' = \langle 105,160,190,235,225 \rangle$. This monoid is isomorphic to the numerical semigroup $H = \langle 21,32,38,47,45 \rangle$; we might wonder if this numerical semigroup satisfies the thesis of our main Theorem. However, this is not the case, as $\PF(H)=\{81,93,103,107,120\}$, and the ideal $I_H$ is minimally generated by $14$ binomials: hence $H$ is not pseudo-symmetric, and $\Bbbk[H]$ is not an almost complete intersection.
\item Let ${\bf a}=(2,2,3,3,5)$. This vector has the same entries as the one in the previous Example; in general, arbitrary permutations in $\Sym(5)$ might change the determinant; compare this with the behavior of the cascade polynomial in Theorem \ref{ThmCascadePolynomial}. In this case we have
\[ \det(\bbM) \cdot \bbM^{-1}{\bf 1}_e =  (237,249,157,154,118)^\intercal,\]

and since $\gcd(237,249,157,154,118)=1$, the main Theorem holds in this case: therefore the hypotheses of Theorem \ref{mainTh} are order sensitive.
\end{enumerate}
\end{ex}

\begin{ex}\label{symnotcrit}
Let $H=\langle 14,15,20,21 \rangle$. Then $I_H=(x_4^2-x_1^3,x_3^3-x_2^4,x_1x_4-x_2x_3)$ is a complete intersection, it has a unique minimal system of binomial generators, and is not critical.
Compare with Theorem \ref{ThmCriticalIffAci}.
\end{ex}

\begin{ex}\label{CInotunique}
Let $H=\langle 6,8,9\rangle$. We have $I_H=(x_1^3-x_3^2,x_2^3-x_1^4)=(x_1^3-x_3^2,x_2^3-x_1x_3^2)$, thus, $I_H$ is a complete intersection but it does not admit a unique minimal system of binomial generators.
\end{ex}

\begin{ex}\label{acinotcrit}
Let $H=\langle 26,46,67,91 \rangle$. 
Then $I_H$ is minimally generated by $\{x_4^2-x_1^7,x_3^4-x_1^5x_2^3,x_2x_3^3-x_1^6x_4,x_2^4-x_1x_3x_4\}$,
in particular, it is an almost complete intersection.
Since $\alpha_1=7$, $\alpha_4=2$ and the only factorizations of $182$ are $182=\alpha_1 n_1 = \alpha_4 n_4$, it follows that $I_H$ cannot be critical.
See the discussion after  Definition \ref{DefCriticalToric}.
\end{ex}

\begin{ex}\label{genbycritnotcrit}
Let $H=\langle 85,105,140,238,357 \rangle$. Then $I_H=(x_1^7-x_4x_5,x_1^7-x_2^3x_3^2,x_2^4-x_3^3,x_4^3-x_5^2)$ is generated by critical binomials, however, it is not critical,
since it is not possible to generate both  $x_1^7-x_4x_5$ and $x_1^7-x_2^3x_3^2$ using a single choice of $F(\bv_1)$.
See the discussion after  Definition \ref{DefCriticalToric}.
\end{ex}

\begin{ex}\label{critnotanychoice}
Let $H=\langle 510,595,630,714,840 \rangle$. Then $I_H$ is critical, minimally generated by  $F_1=x_1^7-x_3^3x_5^2$, $F_2 = x_2^6-x_1^7$, $F_3=-F_5=x_3^4-x_5^3$, $F_4=x_4^5-x_1^7$. However, if we replace $F_1$ with $F'_1=x_1^7-x_2^6 = -F_2$, then $\{F'_1,F_2,F_3,F_4,F_5\}$ do not generate $I_H$, since $F_1 \notin (F_2,F_3,F_4)$.
Compare with Theorem \ref{ThmCriticalIffAci}.
\end{ex}

\begin{ex}\label{ce_RF(f)} 
Let $H = \langle 9,12,13,15,29\rangle$. 
Then $\PF(H) =\{16,32\}$ and  $H$ is pseudo-symmetric. 
The element $16 \in \PF(H)$ admits four distinct $\RF$-matrices, which differ only by their last row (encoding different factorizations of $16+29 \in H$).
Consider  the $\RF$-matrix
\[ \bbM=(m_{ij})=\RF(16)=\begin{pmatrix} -1 & 1 & 1 & 0 & 0 \\
                            0 & -1& 1 & 1 & 0 \\    
                            0 &  0& -1& 0 & 1 \\
                            2 &  0&  1& -1& 0 \\
                            2 &  1 &  0 &  1 & -1       
\end{pmatrix}
\]
Since $\alpha_1=3$ and $\alpha_i \ge 2$ for every $i=2,\ldots,5$, we have  $m_{ij} < \alpha_j$ for every $i,j$, and this is the unique $\RF$-matrix satisfying this property.
Since $\alpha_3=3$ and $0 < m_{13} < \alpha_3-1$, the statement of Proposition \ref{mj=alj-1} does not hold for $H$. 
Indeed, since $\alpha_3=3$, it is easy to verify that the minimal binomial $f=x_3^3- x_2^2x_4\in I_H$ does not appear as an RF-relation of $\bbM$. Moreover, it does not appear as an $\RF$-relation of any $\RF$-matrix of $16$.
Compare with Proposition \ref{RF(f)}.
\end{ex}


\begin{thebibliography}{99}

\bibitem[AV94]{AV} A. Alc\'antar, R.H. Villarreal, \textit{Critical binomials of monomial curves}, Comm. Algebra {\bf 22} (1994), 3037--3052.


\bibitem[BF97]{BF} V. Barucci, R. Fr\"oberg, \textit{One-dimensional almost Gorenstein rings}, J. Algebra \textbf{188} (1997), 418--442.


\bibitem[B75]{B75} H. Bresinsky, \textit{Symmetric semigroups of integers generated by $4$ elements}, Manuscripta Math. \textbf{17} (1975), 205--219.


\bibitem[BE77]{BE}
D. A. Buchsbaum, D. Eisenbud,
\textit{Algebra structures for finite free resolutions, and some structure theorems for ideals of codimension 3},
Amer. J. Math. {\bf 99}  (1977), 447--485.

\bibitem[B68]{Burch}
L. Burch, 
\textit{On ideals of finite homoloǵical dimension in local rings},
Math. Proc. Camb. Philos. Soc. {\bf 64}
(1968), 941--948.

\bibitem[CGNW25]{CGNW} T. Chmiel, L. Guerrieri, X. Ni, J. Weyman, \textit{Structure theorems for Gorenstein ideals of codimension four with small number of generators}, In: Commutative Algebra. The Mathematical Legacy of Wolmer V. Vasconcelos, eds. J. Brennan, A. Simis, de Gruyter (2025), 297--332.

\bibitem[DGM]{DGM} M. Delgado, P.A. Garc\'ia-S\'anchez, J. Morais, NumericalSgps -
a GAP package, Ver. 1.50 (2026). Available at \url{http://www.gap-system.org/Packages/numericalsgps}.

\bibitem[D76]{Delorme} C. Delorme, \textit{Sous-monoïdes d’intersection complète de $ N$}, Ann. Sci. Éc. Norm. Supér. {\bf 9} (1976), 145--154.

\bibitem[El83]{El83} S. Eliahou, \textit{Courbes monomiales et algèbre de Rees symbolique}, Ph.D. Thesis, Universit\'e de Genève (1983).

\bibitem[E04]{Et1} 
K. Eto, \emph{When is a binomial ideal equal to a lattice ideal up to radical? II}, 
Toyama Math. J. {\bf 27} (2004), 111--123.

\bibitem[E17]{E}  K. Eto, \textit{Almost Gorenstein monomial curves in affine four space}. J. Algebra,
\textbf{488} (2017), 362--387.

\bibitem[GAP]{GAP} The GAP Group, GAP -- Groups, Algorithms, and Programming, Ver. 4.16.1 (2026). Available at \url{https://www.gap-system.org}.


\bibitem[GLO21]{GLO}  P. A. Garc\'ia-S\'anchez, D. Llena, I. Ojeda, \textit{Critical binomial ideals of Northcott type}, J. Aust. Math. Soc. {\bf 110} (2021), 48--70.



\bibitem[GO10]{GO} P. A. Garc\'ia-S\'anchez and I. Ojeda, \emph{Uniquely presented finitely generated commutative monoids}, Pacific J. Math.  {\bf 248} (2010), 91–-105.

\bibitem[GTT15]{GTT} S. Goto, R. Takahashi, N. Taniguchi, \textit{Almost Gorenstein rings - towards a theory of higher dimension}, J. Pure Appl. Algebra {\bf 219} (2015), 2666--2712. 

\bibitem[M2]{M2} D. Grayson, M. Stillman, Macaulay2, a software system for research in algebraic geometry, Ver. 1.26.06 (2026). Available at \url{http://www2.macaulay2.com}.

\bibitem[GCNW25]{GCNW} L. Guerrieri, T. Chmiel, X. Ni, J. Weyman, \emph{Grade three perfect ideals and length four self-dual resolutions}, arXiv:2512.01079 (2025).


\bibitem[H70]{H} J. Herzog, \textit{Generators and relations of abelian semigroups and semigroup rings}, Manuscripta Math. \textbf{3} (1970), 175--193.


\bibitem[HW19]{HW} J. Herzog, K. Watanabe, \emph{Almost symmetric numerical semigroups},  
Semigroup Forum {\bf 98} (2019), 589--630. 

\bibitem[HU85]{HU} C. Huneke, B. Ulrich, \emph{Divisor class groups and deformations}, Amer. J. Math. {\bf 107} (1985), 1265--1303.

\bibitem[KO14]{KO} A. Katsabekis, I. Ojeda, \textit{An indispensable classification of monomial curves in $\mathbb{A}^4(k)$}, Pacific. J. Math. {\bf 268} (2014), 96--116.


\bibitem[K74]{Ku74} E. Kunz, \emph{Almost complete intersections are not Gorenstein rings}, J. Algebra {\bf 28} (1974), 111--115.

\bibitem[K93]{Kustin93}
A. R. Kustin, 
\textit{Classification of the Tor-algebras of codimension four almost complete intersections},
Proc. Amer. Math. Soc. {\bf 339} (1993), 61--85.

\bibitem[K94]{Kustin}
A. R. Kustin, 
\emph{The minimal resolution of a codimension four almost complete intersection is a DG-algebra},
J. Algebra {\bf 168}  (1994), 371--399.

\bibitem[K95]{Kustin95}
A. R. Kustin,
\emph{Huneke-Ulrich almost complete intersections of Cohen-Macaulay type two},
J. Algebra {\bf 174}  (1995), 373--429.


\bibitem[M16]{Mo} A. Moscariello, \emph{On the type of an almost Gorenstein monomial curve},
J. Algebra  \textbf{456} (2016), 266--277.

\bibitem[MS25]{MS} A. Moscariello, A. Sammartano,
Open problems on relations of numerical semigroups, In: Recent progress in rings and factorization theory, eds. M. Bresar, A. Geroldinger, B. Olberding, and D. Smertnig, Springer (2025), 365--380.



\bibitem[N13]{N} H. Nari, \textit{Symmetries on almost symmetric numerical semigroups}, Semigroup Forum \textbf{86} (2013), 140--154.




\bibitem[PS93]{Palmer}
S. Palmer Slattery,
\emph{Algebra structures on resolutions of rings defined by grade four almost complete intersections},
J. Algebra {\bf 159} (1993), 1--46.

\bibitem[R97]{R97} J. C. Rosales, \emph{On presentations of subsemigroups of $\N^n$}, Semigroup Forum {\bf 55} (1997), 152--159.



\end{thebibliography}
\end{document}